\documentclass[a4paper,twoside,reqno]{amsart}
\usepackage{amssymb}
\usepackage[mathscr]{eucal}
\usepackage[all]{xy}
\usepackage{graphicx}

\usepackage{hyperref}

\theoremstyle{plain}
\newtheorem{prop}{Proposition}[section]
\newtheorem{thm}[prop]{Theorem}
\newtheorem{cor}[prop]{Corollary}
\newtheorem{lem}[prop]{Lemma}

\theoremstyle{definition}

\newtheorem{examples}[prop]{Examples}
\newtheorem{rem}[prop]{Remark}
\newtheorem{rems}[prop]{Remarks}
\newtheorem{example}[prop]{Example}

\newtheorem{lab}[prop]{}

\newcommand{\A}{{\mathbb{A}}}
\newcommand{\C}{{\mathbb{C}}}
\newcommand{\N}{{\mathbb{N}}}
\newcommand{\R}{{\mathbb{R}}}
\newcommand{\Z}{{\mathbb{Z}}}

\newcommand{\m}{{\mathfrak{m}}}
\newcommand{\p}{{\mathfrak{p}}}
\newcommand{\q}{{\mathfrak{q}}}

\newcommand{\scrF}{{\mathscr{F}}}
\newcommand{\scrO}{{\mathscr{O}}}
\newcommand{\scrP}{{\mathscr{P}}}
\newcommand{\scrX}{{\mathscr{X}}}

\newcommand{\x}{{\mathtt{x}}}

\DeclareMathOperator{\cone}{cone}
\DeclareMathOperator{\conv}{conv}
\DeclareMathOperator{\Gal}{Gal}
\DeclareMathOperator{\Hom}{Hom}
\DeclareMathOperator{\im}{im}

\DeclareMathOperator{\ord}{ord}
\DeclareMathOperator{\Pic}{Pic}
\DeclareMathOperator{\Quot}{Quot}
\DeclareMathOperator{\Sat}{Sat}
\DeclareMathOperator{\spn}{span}
\DeclareMathOperator{\Spec}{Spec}
\DeclareMathOperator{\Sper}{Sper}
\DeclareMathOperator{\supp}{supp}

\newcommand{\Ex}{\mathrm{Ex}}

\newcommand{\interior}{\mathrm{int}}
\newcommand{\rc}{\mathrm{rc}}

\newcommand{\labelto}[1]{\overset{#1}{\longrightarrow}}

\newcommand{\wt}[1]{\widetilde{#1}}
\renewcommand{\subset}{\subseteq}
\renewcommand{\supset}{\supseteq}

\newcommand{\ol}[1]{\overline{#1}}
\newcommand{\plus}{{\scriptscriptstyle+}}
\newcommand{\all}{\forall\>}
\newcommand{\ex}{\exists\>}
\newcommand\lal{\looparrowleft}

\renewcommand{\emptyset}{\varnothing}
\renewcommand{\setminus}{\smallsetminus}
\renewcommand{\epsilon}{\varepsilon}

\newcommand{\bil}[2]{\langle{#1},{#2}\rangle}
\newcommand{\comp}{\mathbin{\scriptstyle\circ}}

\newcommand{\sa}{semi-algebraic}

\begin{document}

\title
[Semidefinite representation for convex hulls of curves]
{Semidefinite representation for convex hulls\\of real algebraic
  curves}

\author{Claus Scheiderer}
\address
  {Fachbereich Mathematik und Statistik \\
  Universit\"at Konstanz \\
  D--78457 Konstanz \\
  Germany}
\email{claus.scheiderer@uni-konstanz.de}
\urladdr{http://www.math.uni-konstanz.de/\textasciitilde scheider}

\begin{abstract}
We show that the closed convex hull of any one-dimensional \sa\
subset of $\R^n$ has a semidefinite representation, meaning that it
can be written as a linear projection of the solution set of some
linear matrix inequality.
This is proved by an application of the moment relaxation method.
Given a nonsingular affine real algebraic curve $C$ and a compact
\sa\ subset $K$ of its $\R$-points, the preordering $\scrP(K)$ of all
regular functions on $C$ that are nonnegative on $K$ is known to be
finitely generated. We prove that $\scrP(K)$ is stable, meaning that
uniform degree bounds exist for weighted sum of squares
representations of elements of $\scrP(K)$.
We also extend this last result to the case where $K$ is only
virtually compact.
The main technical tool for the proof of stability is the archimedean
local-global principle.
As a consequence of our results we prove that every convex \sa\
subset of $\R^2$ has a semidefinite representation.
\end{abstract}



\maketitle


\section*{Introduction}

Let $K\subset\R^n$ be a real algebraic set, or more generally a \sa\
set. The question of how to represent the convex hull $\conv(K)$ of
$K$ has attracted growing attention in recent years. A good part of
this interest originates from optimization theory, namely from the
problem of optimizing a linear functional over $K$. One of the most
promising approaches that have been discussed is to express
$\conv(K)$ (at least up to taking closures) as a linear projection of
a spectrahedron, that is, of a set described by a linear matrix
inequality. In other words, one would like to find symmetric real
matrices $M_i$, $N_j$ of some size (for $0\le i\le n$, $1\le j\le k$
and some~$k$) such that, writing
\begin{equation}\label{afflinsyst}%
M(x,y)\>=\>M_0+\sum_{i=1}^nx_iM_i+\sum_{j=1}^ky_jN_j,
\end{equation}
the closure of $\conv(K)$ coincides with the closure of the set
\begin{equation}\label{sdpset}%
S\>=\>\{x\in\R^n\colon\>\ex y\in\R^k\ M(x,y)\succeq0\}.
\end{equation}
Here $M\succeq0$ means that the symmetric matrix $M$ is positive
semidefinite. In view of the very efficient methods available in
semidefinite programming, such a representation is perfectly well
suited for optimizing linear functionals over~$K$.

Another approach tries to understand the set $\conv(K)$ via the dual
algebraic variety of the Zariski closure of its boundary, see
\cite{rs1}, \cite{rs2}, \cite{sist} for more details.

A subset $S\subset\R^n$ is said to be \emph{semidefinitely
representable} (or to be sdp-representable, or to have a semidefinite
representation), if it can be written as in \eqref{afflinsyst} and
\eqref{sdpset} with suitable symmetric matrices $M_i,\,N_j$. The
question of characterizing sdp-representable sets
was raised by Nemirovski in his plenary address at the ICM in Madrid
\cite{nm}.
Any sdp-representable set is clearly \sa\ and convex, and for many
years no other restriction was known. In 2009, Helton and Nie
\cite{hn1} conjectured that conversely every convex \sa\ set has a
semidefinite representation. This conjecture was recently disproved
by the author \cite{sch:sdp}.
In the present paper however, we prove the existence of a
semidefinite representation for the closed convex hull of any
one-dimensional \sa\ set in $\R^n$. Using this result, we show that
every convex \sa\ subset of the plane has a semidefinite
representation, i.e.\ we show that the Helton-Nie conjecture does
hold in dimension two.

Our result does not extend to convex hulls of sets of dimension
greater than one. Indeed, for \emph{every} \sa\ set $K\subset\R^n$ of
dimension at least two, there exists a polynomial map $\varphi\colon
\R^n\to\R^N$ (for some $N\ge1$) such that the closed convex hull of
$\varphi(K)$ in $\R^N$ has no semidefinite representation. This is
proved in \cite{sch:sdp}.

For the construction of semidefinite representations we use the
moment relaxation method, introduced by Lasserre and Parrilo
(\cite{la}, \cite{pa1}, \cite{pa2}, see also \cite{la:book},
\cite{bpt}).
Computing the convex hull of a set $K\subset\R^n$ (that we assume to
be basic closed \sa) means to determine the linear moments of all
probability measures on $K$ for which these moments exist.
By considering finite-dimensional relaxations of the $K$-moment
problem,
one obtains a nested hierarchy $K(1)\supset K(2)\supset\cdots$ of
explicitly sdp-represented sets that all contain $K$. Their closures
$\ol{K(d)}=TH(d)$ have also been studied under the name \emph{theta
bodies} of~$K$
(see \cite{gpt} and \cite{bpt}, ch.~7). When $K$ is a compact \sa\
set, the sets $K(d)$ approximate $\conv(K)$ arbitrarily closely.
Moreover, the approximation becomes exact, that is, $K(d)=\conv(K)$
for some $d\ge1$, if and only if every linear polynomial that is
nonnegative on $K$ has a weighted sum of squares representation with
uniform degree bounds on the summands. See Theorem \ref{relaxsumm}
below for a rigorous formulation.

We consider a nonsingular affine algebraic curve $C$ over $\R$ and a
compact \sa\ subset $K$ of $C(\R)$, the set of real points on $C$. We
work in $\R[C]$, the affine coordinate ring of $C$. Let $\scrP(K)$ be
the saturated preordering of $K$, i.e.,
the set of all elements of $\R[C]$ that are nonnegative on $K$. It is
known \cite{sch:mz} that $\scrP(K)$ is finitely generated as a
preordering. This means that there exist finitely many elements
$1=h_0,\,h_1,\dots,h_r\in\scrP(K)$ such that every $f\in\scrP(K)$ has
a representation
\begin{equation}\label{pkrep}%
f\>=\>\sum_{i=0}^r\sum_jp_{ij}^2h_i
\end{equation}
with $p_{ij}\in\R[C]$. Fixing $C$, $K$ and the $h_i$, the main result
of this paper (Corollary \ref{mainstab}) says that there exist
uniform degree bounds for such representations. That is, every
$f\in\scrP(K)$ has some representation \eqref{pkrep} in which the
degrees of the summands are bounded above by some number that depends
only on $\deg(f)$.
(We are using degrees here to simplify the exposition, and so we
tacitly assume that $C$ is considered with a fixed embedding in some
affine space.) Technically, this result is expressed by saying that
the preordering $\scrP(K)$ is stable.
From this it follows that, for any morphism $\varphi\colon C\to\A^n$
into affine space of any dimension, the relaxation process for the
convex hull of $\varphi(K)$ in $\R^n$ becomes exact. In fact, this
latter property is equivalent to stability of $\scrP(K)$.

Our method for proving stability of $\scrP(K)$ may be of interest in
that we do not show the existence of degree bounds directly. Rather,
we establish the following equivalent fact: For any real closed field
$R$ containing $\R$, the preordering generated by the $h_i$ in
$R[C]=\R[C]\otimes R$ is again saturated (Theorem \ref{main}). This
fact, in turn, is proved by an application of the archimedean
local-global principle \cite{sch:surf}, which allows us to reduce the
problem to local rings. At first sight this may seem impossible since
the field $R$ is non-archimedean. We get around this by
working in the ring $B[C]=\R[C]\otimes B$, rather than in $R[C]$,
where $B$ is the smallest convex subring of $R$ that contains~$\R$
(so $B$ is a non-noetherian valuation ring). We believe that this way
of applying the local-global principle is novel and perhaps somewhat
unexpected.

In the case where $C$ has genus one and $K=C(\R)$ is the full real
curve (assumed to be compact), our main result was already known by
\cite{sch:conv}. In that paper, using geometric arguments of
Riemann-Roch type, we had given degree bounds of quite explicit
nature, resulting in bounds for the sizes of the derived exact
semidefinite representations.
For all curves of higher genus, as well as for genus one and $K\ne
C(\R)$, our results are new. In contrast to the method used in
\cite{sch:conv}, the techniques used in the present paper do
unfortunately not seem to give any explicit degree bounds.

From Corollary \ref{mainstab} we deduce the existence of a
semidefinite representation for the convex hull of any compact \sa\
set $S\subset\R^n$ with $\dim(S)\le1$ (Theorem \ref{sdp}). For this
one first desingularizes via normalization and then uses the moment
relaxation process. This case in turn implies the existence of such a
representation for the closed convex hull of any \sa\ set $S$ with
$\dim(S)\le1$, not necessarily compact (Theorem \ref{sdpclosed}).
From this we establish the Helton-Nie conjecture in dimension two
(Theorem \ref{hndim2}).

On the other hand, we extend the stability result to certain
noncompact cases. Namely, when $C$ is a nonsingular affine curve and
$K\subset C(\R)$ is a closed \sa\ set that is merely virtually
compact (meaning that there exists $f\in\R[C]$ that is nonconstant
and bounded on $K$),
the saturated preordering $\scrP(K)$ is still finitely generated and
stable (Theorem \ref{vcptmain}). Again, this is proved by a reduction
to the compact case.

We would like to point out that, by the results of \cite{sch:stab},
our main result on degree bounds for (weighted) sum of squares
representations does not extend to dimensions bigger than one. For
example, it was shown there for any nonsingular affine $\R$-variety
$V$ with
$V(\R)\ne\emptyset$ compact and $\dim(V)\ge2$, that degree bounds for
sums of squares in $\R[V]$ cannot exist.

For practical matters our results imply the following. Suppose we are
given a compact \sa\ set $K\subset\R^n$, $\dim(K)=1$, and a
polynomial $f\in\R[x]=\R[x_1,\dots,x_n]$, and want to find $f_*:=\min
f(K)$. For simplicity assume that $K=C(\R)$ is a real algebraic curve
without singularities (the more general case can be reduced to this
one).
For every degree $d$ consider
$$c_d\>:=\>\max\bigl\{c\in\R\colon f-c\text{ is a sos of polynomials
of }\deg\le d\text{ modulo }I_C\bigr\}$$
($I_C:={}$ideal of $C$ in $\R[x]$). Then $c_d$ is the optimum of an
explicit semidefinite program, and $c_d\uparrow f_*$ by the general
results of \cite{la}. Our results imply that we have in fact finite
convergence, i.e.\ $f_*=c_d$ for some $d\in\N$ which \emph{depends
only on $C$ and $\deg(f)$}, but not on $f$. If $C$ has genus $g\le1$,
upper bounds for $d$ are known explicitly (\cite{pa2}, \cite{he} for
$g=0$ and \cite{sch:conv} for $g=1$), but unfortunately not
otherwise.

Both for theoretical and practical reasons it would be highly
desirable to have a more constructive approach to the results of this
paper. In particular, one would like to have some information on the
nature of the degree bounds whose existence is proved here.

The paper is organized as follows. In Section~\ref{sect:relaxrappl}
we give a brief account of the relaxation method for constructing
semidefinite representations of convex hulls, in the generality that
is needed here. Section~\ref{sect:aux} contains auxiliary results for
working in the ring $\R[C]\otimes B$. This ring plays a key role in
the proof of stability of $\scrP(K)$ in the compact case
(Section~\ref{sect:main}). The existence of semidefinite
representations for compact convex hulls is deduced in
Section~\ref{sect:sdp}, and the extension to closed convex hulls of
arbitrary one-dimensional sets is discussed in
Section~\ref{sect:noncpt}. Finally, Section~\ref{sect:virtcpt}
contains the proof of stability in the virtually compact case.

This paper was originally written and submitted to a journal in 2012,
but eventually got rejected. At that time the Helton-Nie conjecture
was still open, and the results of this paper were considered as
additional support for this conjecture. The present form is a
slightly revised and updated version.
\smallskip

\textbf{Acknowledgement.}
This research was supported by DFG grant \texttt{SCHE281/10-1}. I am
grateful to Tim Netzer for his useful remarks on a preliminary 2012
version. His comments led to substantial improvements of some of the
initial results.


\section{Notations and preliminaries}

\begin{lab}
Let $k$ be a field. By an algebraic $k$-variety (or simply
$k$-variety) we mean a reduced and separated $k$-scheme of finite
type. Most algebraic varieties and schemes in this paper will be
affine. An affine $k$-variety is therefore the Zariski spectrum
$V=\Spec(A)$ of a $k$-algebra~$A$ which is finitely generated and
reduced (no nonzero nilpotent elements). Following common practice,
we also write $A=k[V]$ and call this ring the affine coordinate ring
of $V$. If $E$ is any $k$-algebra, then $V(E)=\Hom_k(A,E)$ denotes
the set of $E$-valued points of $V$. Given $\xi\in V(E)$ and $f\in
A$, we usually write $f(\xi)$ (rather than $\xi(f)$) for the result
of evaluating the homomorphism $\xi$ on~$f$.

A curve over $k$ is a $k$-variety all of whose irreducible components
have dimension one. An affine curve $C$ over $k$ is irreducible
(resp., irreducible and nonsingular) if and only if the ring $k[C]$
is an integral domain (resp., a Dedekind domain).
\end{lab}

\begin{lab}\label{sperrappl}%
We need to employ the real spectrum, and we briefly recall the basic
notions. See \cite{bcr}, \cite{pd}, \cite{ma} or \cite{sch:guide} for
full details and background. All rings are assumed to be commutative
and to have a unit. The real spectrum of the ring $A$, denoted
$\Sper(A)$, is the set consisting of all pairs $\alpha=(\p,\omega)$
where $\p\in\Spec(A)$ and $\omega$ is an ordering of the residue
field of $\p$. The prime ideal $\p$ is called the support of
$\alpha$, written $\p=\supp(\alpha)$.

For $f\in A$ and $\alpha=(\p,\omega)\in\Sper(A)$, the notation
``$f(\alpha)\ge0$'' (resp., ``$f(\alpha)>0$'') indicates that the
residue class $f$~mod~$\p$ is non-negative (resp., positive) with
respect to~$\omega$. The \emph{(Harrison) topology} on $\Sper(A)$ is
defined to have the collection of sets $U(f)=\{\alpha\in\Sper(A)
\colon f(\alpha)>0\}$, $f\in A$, as a subbasis of open sets. The
support map $\supp\colon\Sper(A)\to\Spec(A)$ is continuous. A subset
of $\Sper(A)$ is called \emph{constructible} if it is a finite
boolean combination of sets $U(f)$, $f\in A$, that is, if it can be
described by imposing sign conditions on finitely many elements of
$A$. Given $\alpha,\,\beta\in\Sper(A)$, one says that $\alpha$
specializes to $\beta$ (or that $\beta$ is a specialization of
$\alpha$) if $\beta$ lies in $\ol{\{\alpha\}}$, the closure of the
set $\{\alpha\}$. Any ring homomorphism
$\varphi\colon A\to B$ induces a continuous map $\varphi^*\colon
\Sper(B)\to\Sper(A)$ in a natural and functorial way.

A convenient alternate way to think of the real spectrum is to
observe that every point of $\Sper(A)$ is represented by a ring
homomorphism $A\to R$ into some real closed field $R$. Two
homomorphisms $A\to R_i$ ($i=1,2$) represent the same point of
$\Sper(A)$ if and only if there exists a third homomorphism $A\to R$
into a real closed field $R$ together with $A$-embeddings $R_i\to R$
($i=1,2$).
\end{lab}

\begin{lab}\label{quadmodsorit}%
Let $A$ be a ring. By $\Sigma A^2$ we denote the set of (finite)
sums of squares in $A$. A subset $M\subset A$ is called a quadratic
module of $A$ if $1\in M$, $M+M\subset M$ and $a^2M\subset M$ for
every $a\in A$ hold. If in addition $MM\subset M$ holds then $M$ is
called a preordering of $A$.

A quadratic module $M$ is finitely generated if there exist finitely
many elements $h_1,\,\dots,\,h_r\in M$ such that (putting $h_0:=1$)
$$M\>=\>(\Sigma A^2)h_0+\cdots+(\Sigma A^2)h_r\>:=\>
\Bigl\{\sum_{i=0}^rs_ih_i\colon s_0,\dots,s_r\in\Sigma A^2\Bigr\}.$$
We say in this case that the quadratic module $M$ is generated by
$h_1,\dots,h_r$.

A quadratic module $M$ of $A$ is said to be archimedean if $\Z+M=A$,
or equivalently, if for every $a\in A$ there exists a positive
integer $n$ such that $n\pm a\in M$.

Given a quadratic module $M\subset A$, one associates with $M$ the
closed subset $\scrX_M:=\{\alpha\in\Sper(A)\colon f(\alpha)\ge0$ for
every $f\in M\}$ of $\Sper(A)$. The saturation of $M$ is the
preordering $\Sat(M):=\{f\in A\colon f\ge0$ on $\scrX_M\}$ of $A$.
The quadratic module $M$ is called saturated if $M=\Sat(M)$.
Any of \cite{pd}, \cite{ma} or \cite{sch:guide} contains more
background on quadratic modules or preorderings and their
saturations.

The notion of stability for a quadratic module is basic for this
paper. It will be recalled in \ref{stablrappl}.
\end{lab}

\begin{lab}\label{tilderappl}%
Let $R$ be a real closed field, and let $V$ be an affine $R$-variety.
Given a \sa\ set $K\subset V(R)$, we denote the associated
constructible subset of $\Sper R[V]$ by $\wt K$, see \cite{bcr} 7.2.
Given any finite system of inequalities that describes $K$, the set
$\wt K$ is the subset of $\Sper R[V]$ that is described by the same
system. The saturated preordering associated with $K$ is denoted
$\scrP(K)$, that is,
$$\scrP(K)\>=\>\{f\in R[V]\colon f|_K\ge0\}.$$
\end{lab}

\begin{example}
Let $h_1,\dots,h_r\in\R[x_1,\dots,x_n]$, and consider the basic
closed set
$$K\>=\>\bigl\{\xi\in\R^n\colon h_1(\xi)\ge0,\,\dots,\,h_r(\xi)\ge0
\}$$
in $\R^n$. The quadratic module $M$ generated by $h_1,\dots,h_r$
satisfies $M\subset\scrP(K)$. In general equality does not hold,
i.e., there exist polynomials $f$ with $f|_K\ge0$ but $f\notin M$.
If $M$ is archimedean then $M$ contains every polynomial $f$ with
$f|_K>0$, by the archimedean Positivstellensatz (see \cite{pu},
\cite{pd} or \cite{ma}). Note that $M$ archimedean implies that $K$
is compact. Conversely, if $K$ is compact, and if $M$ is a
preordering, then $M$ is archimedean (Schm\"udgen Positivstellensatz,
\cite{sm}, \cite{pd} or \cite{ma}).
\end{example}

\begin{lab}\label{convrappl}%
The convex hull of a set $S\subset\R^n$ is denoted $\conv(S)$. If
$K\subset\R^n$ is a closed convex set, a point $a\in K$ is called an
extreme point of $K$ if $a=(1-t)b+tc$, where $b,\,c\in K$ and
$0<t<1$, implies $b=c=a$. The set of extreme points of $K$ is
denoted $\Ex(K)$. When $K$ is a \sa\ set, the set $\Ex(K)$ is \sa\
as well.
\end{lab}


\section{The relaxation method}\label{sect:relaxrappl}%

\begin{lab}\label{stablrappl}%
Let $A$ be a finitely generated $\R$-algebra, and let $M$ be a
finitely generated quadratic module in $A$, say $M=\Sigma_Ah_0+\cdots
+\Sigma_Ah_r$ with $1=h_0,\,h_1,\dots,h_r\in A$ and $\Sigma_A:=\Sigma
A^2$ (the cone of sums of squares in $A$). The quadratic module $M$
is said to be \emph{stable} (see \cite{ps}, \cite{sch:stab}) if,
given any finite-dimensional linear subspace $U$ of $A$, there exists
a finite-dimensional linear subspace $W$ of $A$ with
$$M\cap U\>\subset\>\Sigma_Wh_0+\cdots+\Sigma_Wh_r.$$
Here $\Sigma_W$ denotes the set of sums of squares of elements of $W$.
The property of being stable does not depend on the choice of the
generators $h_0,\dots,h_r$ of $M$. If $A$ is a polynomial ring over
$\R$, stability of $M$ means that there exists a map $\varphi\colon
\N\to\N$ such that, for every $f\in M$, there exists a representation
$f=\sum_{i,j}p_{ij}^2h_i$ with suitable polynomials $p_{ij}$ such
that $\deg(p_{ij}^2h_i)\le\varphi(\deg(f))$ for all $i,j$.
\end{lab}

\begin{lab}
By a \emph{semidefinite representation} of a set $S\subset\R^n$ one
means a representation of $S$ in the form
$$S\>=\>\Bigl\{x\in\R^n\colon\ex y\in\R^k\ M_0+\sum_{i=1}^nx_iM_i
+\sum_{j=1}^ky_jN_j\succeq0\Bigr\}$$
with suitable $k\ge0$ and real symmetric matrices $M_i,\,N_j$ of some
size.
A~set $S$ that has a semidefinite representation is also said to be
\emph{semidefinitely representable}, or \emph{sdp-representable}.
Other terms often used in the literature are projected spectrahedron,
spectrahedral shadow, or lifted LMI-representable set.
\end{lab}

We now recall the method of moment relaxation \cite{la} for
constructing semidefinite representations, in a generality adapted to
our needs. For more background we refer to Chapter~11 of
\cite{la:book}, and to Chapters 6 and~7 of \cite{bpt}. We only
outline the basic principle of the construction, ignoring possible
refinements.

\begin{lab}\label{relaxrappl}%
Let $A$ be a finitely generated reduced $\R$-algebra. We denote the
associated affine $\R$-variety by $V=\Spec(A)$, so $A=\R[V]$,
and we always equip the set $V(\R)=\Hom(A,\R)$ of real points of $V$
with its natural Euclidean topology. Fix elements $1=h_0,\,h_1,\dots,
h_r\in A$, write $\Sigma_A:=\Sigma A^2$ for the cone of sums of
squares in $A$, let
$$M\>=\>h_0\Sigma_A+\cdots+h_r\Sigma_A$$
be the quadratic module in $A$ generated by the $h_i$, and let
$$K\>=\>\{\xi\in V(\R)\colon h_1(\xi)\ge0,\dots,h_r(\xi)\ge0\}$$
be the associated basic closed \sa\ subset of $V(\R)$. We assume that
$K$ is Zariski dense in $V$. Fix a finite-dimensional linear subspace
$L\subset A$ containing~$1$, and let $1,\,x_1,\dots,x_n$ be a basis
of $L$. We consider the morphism $\varphi=\varphi_L=(x_1,\dots,x_n)$
from $V$ to affine $n$-space determined by $L$, and the induced map
$\varphi\colon V(\R)\to\R^n$.

Given a linear subspace $B\subset A$ we denote by $BB$ the linear
subspace of $A$ spanned by all products $b_1b_2$ with $b_1,\,b_2\in
B$. Fix a tuple $W=(W_0,\dots,W_r)$ of finite-dimensional linear
subspaces of $A$, and consider the linear subspace
$$U\>:=\>W_0W_0+h_1W_1W_1+\cdots+h_rW_rW_r$$
of $A$. We assume that $L$ is contained in $U$, and we denote by
$\rho\colon U'\to L'$ the restriction map between the dual linear
spaces. By $U'_1$ (resp.\ $L'_1$) we denote the set of all linear
forms $\lambda$ in $U'$ (resp.\ in $L'$) with $\lambda(1)=1$, and we
identify $\R^n$ with $L'_1$ via the map
$$L'_1\>\labelto\sim\>\R^n,\quad\lambda\>\mapsto\>\bigl(\lambda
(x_1),\dots,\lambda(x_n)\bigr).$$
For $i=0,\dots,r$ let $\Sigma_{W_i}\subset W_iW_i$ denote the cone of
sums of squares of elements of $W_i$. The set
$$M_W\>:=\Sigma_{W_0}+h_1\Sigma_{W_1}+\cdots+h_r\Sigma_{W_r}$$
is contained in $M\cap U$ and is a convex \sa\ cone in $U$. Since $K$
is Zariski dense in $V$, we have $M\cap(-M)=\{0\}$. This implies that
$M_W$ is closed in $U$ (\cite{ps} Prop.\ 2.6).
Let $M_W^*\subset U'$ be the dual cone of $M_W$. Then $M_W^*$ can
be defined by a (homogeneous) linear matrix inequality, that is,
$M_W^*$ is a spectrahedral cone in $U'$. The subset $M_W^*\cap U'_1$
of $M_W^*$ is therefore a spectrahedron as well. Its image set
$$K_W\>:=\>\rho(M_W^*\cap U'_1)\>=\>L'_1\cap\rho(M_W^*)\>\subset\>
\R^n$$
under the restriction map $\rho\colon U'_1\to L'_1=\R^n$ is therefore
an sdp-representable set by construction. For every $\xi\in K$, the
cone $M_W^*$ contains the evaluation map at $\xi$ (restricted to
$U$). Therefore $K_W$ contains the set $\varphi(K)$, and therefore we
have $\conv(\varphi(K))\subset K_W$. Increasing the subspaces
$W_0,\dots,W_r$ of $A$ results in making the set $K_W$ smaller. The
main facts are summarized in the following theorem (c.f.\ \cite{la}
Theorem~2):
\end{lab}

\begin{thm}\label{relaxsumm}%
Let $L\subset A$ be a fixed linear subspace with basis $1,x_1,\dots,
x_n$, and let $\varphi\colon V\to\A^n$ be the associated morphism.
With assumptions and notation from \ref{relaxrappl}, we have:
\begin{itemize}
\item[(a)]
$\ol{K_W}\>=\>\{\eta\in\R^n\colon\all f\in L\cap M_W$ $f(\eta)
\ge0\}$;
\item[(b)]
the inclusion $\ol{\conv(\varphi(K))}\subset\ol{K_W}$ of closed
convex sets is an equality if and only if $L\cap\scrP(K)\subset
M_W$;
\item[(c)]
if $M$ is archimedean (see \ref{quadmodsorit}) then
$\conv(\varphi(K))=\bigcap_WK_W$, intersection over all systems
$W=(W_0,\dots,W_r)$ of finite-dimensional subspaces of $A$.
\end{itemize}
\end{thm}

If $K$ is compact then $\conv(\varphi(K))$ is again compact by
Carath\'eodory's lemma, and for any fixed tuple $W$ as above we get:

\begin{cor}\label{relaxcpt}%
If $K$ is compact, then $\conv(\varphi(K))=K_W$ holds if and only if
$L\cap\scrP(K)\subset M_W$.
\end{cor}

\begin{lab}\label{relaxexact}%
The moment relaxation for the closed convex hull
$\ol{\conv(\varphi(K))}$ is said to become \emph{exact} if the
equality $\ol{\conv(\varphi(K))}=\ol{K_W}$ holds for some choice
$W=(W_0,\dots,W_r)$ of finite-dimensional subspaces. When $K$ is
compact, this is equivalent to $\conv(\varphi(K))=K_W$.

If one is aiming at describing the convex hull of $\varphi(K)$ in
$\R^n$, approximately or exactly, note that there is a two-fold
freedom of modifying the above construction. On the one hand, we may
enlarge the subspaces $W_0,\dots,W_r$. We may as well enlarge the
quadratic module $M$ by adding finitely many more generators $h_i$
from $\scrP(K)$. Both steps result in making the approximation
tighter. When the saturated preordering $\scrP(K)$ itself is finitely
generated, then choosing $M=\scrP(K)$ will give the closest
approximations for $\conv(\varphi(K))$.

When $K=V(\R)$ is a real algebraic set, and when an embedding
$V\subset\A^n$ is fixed, the closed convex sets $\ol{K_W}\subset\R^n$
resulting from taking $M=\Sigma\R[V]^2$ approximate the closed convex
hull $\ol{\conv(V(\R))}$. Under the name \emph{theta bodies} of $V$
they have been studied by Gouveia, Parrilo, Thomas and others (see
\cite{gpt} and \cite{bpt}, Chapter~7).

Varying the embedding $\varphi$, we see:
\end{lab}

\begin{cor}\label{allembiffstabl}%
Let $V$ be an affine $\R$-variety, let $K\subset V(\R)$ be a basic
closed set, Zariski dense in $V$, and assume that the saturated
preordering $\scrP(K)$ in $\R[V]$ is finitely generated. Then the
following two conditions are equivalent:
\begin{itemize}
\item[(i)]
For any $n\in\N$ and any morphism $\varphi\colon V\to\A^n$ of
$\R$-varieties, the moment relaxation for the closed convex hull
$\ol{\conv(\varphi(K))}$ becomes exact (\ref{relaxexact});
\item[(ii)]
the preordering $\scrP(K)$ in $\R[V]$ is stable (\ref{stablrappl}).
\end{itemize}
\end{cor}

\begin{proof}
After fixing a finite description $\scrP(K)=h_0\Sigma+\cdots+h_r
\Sigma$ (with $\Sigma=\Sigma\R[V]^2$), stability of $\scrP(K)$ means
that, for every finite-dimensional subspace $L\subset\R[V]$
containing~$1$, there exists a tuple $W=(W_0,\dots,W_r)$ of
finite-dimensional subspaces such that $L\cap\scrP(K)\subset M_W$. By
\eqref{relaxsumm}(b), it is equivalent that $\ol{\conv(\varphi(K))}=
\ol{K_W}$, where $\varphi$ is the morphism associated with $L$.
Having $L$ range over all finite-dimensional subspaces means to have
$\varphi$ range over all morphisms from $V$ to affine space of
arbitrary dimension.
Therefore, (i) and (ii) are equivalent.
\end{proof}


\section{Auxiliary results}\label{sect:aux}%

Let $C$ be a nonsingular curve over $\R$. Here we collect results
that are needed for working in the base extension of $C$ to a real
closed valuation ring $B\supset\R$. The situation has some
resemblance to arithmetic surfaces. The main result that will be
needed in the next section is Proposition \ref{nstordspez}.

\begin{lab}\label{wsorit}%
The following setup will be fixed for the entire section. Let $R$ be
a real closed field containing $\R$, the field of real numbers. The
unique ordering of $R$ is denoted $\le$. Let
$$B\>:=\>\bigl\{b\in R\colon\ex n\in\N\ -n<b<n\bigr\}$$
be the convex hull of $\R$ in $R$. Then $B$ is a valuation ring with
quotient field $R$, and we denote by $v\colon R\to\Gamma\cup
\{\infty\}$ the associated Krull valuation. The maximal ideal of $B$
will be denoted by $\m$. The residue field is $B/\m=\R$.

Let $A$ be a finitely generated $\R$-algebra, and write $A_B=A\otimes
B$ and $A_R=A\otimes R$ (with $\otimes:=\otimes_\R$ always). Given
$0\ne f\in A_R$, we can write $f=\sum_{i=1}^ra_i\otimes b_i$ with
$a_i\in A$ and $b_i\in R$ in such a way that $a_1,\dots,a_r$ are
linearly independent over $\R$. Putting
$$w(f)\>:=\>\min\{v(b_i)\colon i=1,\dots,r\}$$
and $w(0):=\infty$ gives a well-defined map $w\colon A_R\to\Gamma\cup
\{\infty\}$ that extends the valuation~$v$. (To see that $w$ is
well-defined, let $f=\sum_{j=1}^sa'_j\otimes b'_j$ be a second
representation with $a'_1,\dots,a'_s$ $\R$-linearly independent. Then
$b_1,\dots,b_r$ and $b'_1,\dots,b'_s$ span the same $\R$-linear
subspace of $R$,
so we can write $b'_j=\sum_ic_{ij}b_i$ with $c_{ij}\in\R$. It follows
that $\min_jv(b'_j)\ge\min_iv(b_i)$. By symmetry, the opposite
inequality holds as well.)
For $f,\,g\in A_R$, it is easy to see that
$w(f+g)\ge\min\{w(f),\,w(g)\}$
and
$w(fg)\>\ge\>w(f)+w(g)$ hold.
For $b\in R$ we moreover have $w(bf)=w(f)+v(b)$.

The residue map $B\to B/\m=\R$ will be denoted by either $b\mapsto
\pi(b)$ or $b\mapsto\ol b$. Accordingly we often denote the induced
homomorphism $A_B\to A$ by $f\mapsto\ol f$. We have $A_B=\{f\in
A_R\colon w(f)\ge0\}$, and for $f\in A_B$ we have $\ol f=0$ iff
$w(f)>0$.
\end{lab}

\begin{lem}\label{geomintbew}%
Assume that the $\R$-algebra $A$ is an integral domain. Then $w(fg)=
w(f)+w(g)$ holds for all $f$, $g\in A_R$, and so $w$ extends to a
valuation of $\Quot(A_R)$, the field of fractions of $A_R$.
\end{lem}

Clearly, the residue field of the valuation $w$ of $\Quot(A_R)$ is
$\Quot(A)$.

\begin{proof}
Since $A$ is a domain, and since $\R$ is relatively algebraically
closed in $R$, the tensor product $A_R$ is a domain, too.
We can write $f=af_0$ and $g=bg_0$ with $a,\,b\in R$ where
$f_0,\,g_0\in A_B$ satisfy $w(f_0)=w(g_0)=0$. So we can assume
$w(f)=w(g)=0$, which means $\ol f,\,\ol g\ne0$ in $A$. Since $A$ is a
domain we have $\ol f\cdot\ol g\ne0$, which implies $w(fg)=0$. The
lemma is proved.
\end{proof}

\begin{lab}\label{krprep}%
Let $A$ be a finitely generated reduced $\R$-algebra,
as before, and write $V=\Spec(A)$ for the affine $\R$-variety
associated with $A$. We need to work with the real spectrum of
$A_B=A\otimes B$. As a set, $\Sper(A_B)$ can be identified with the
disjoint union of the real spectra of the rings $A\otimes R(\q)$,
where $\q$ is a prime ideal of $B$ and $R(\q)$ denotes the residue
field of $\q$  (a~real closed field extension of~$\R$). Given any
point $\xi\in V(\C)=\Hom_\R(A,\C)$, we consider the homomorphism
$$\xi\otimes\pi\colon A\otimes B\to\C,\quad a\otimes b\>\mapsto\>
a(\xi)\ol b$$
and denote its kernel by $M_\xi$.
So
$$M_\xi\>:=\>\Bigl\{\sum_ia_i\otimes b_i\in A\otimes B\colon\>\sum_i
a_i(\xi)\ol b_i=0\text{ in }\C\Bigr\}.$$
Clearly, $M_\xi$ is a maximal ideal of $A\otimes B$ whose residue
field is the residue field of $\xi$ (hence $\R$ or~$\C$).
When $\xi$ is real, i.e.\ $\xi\in V(\R)$, there is a unique point in
$\Sper(A\otimes B)$ whose support is $M_\xi$. This point will be
denoted $\alpha_\xi$.
Conversely, any point $\alpha\in\Sper(A\otimes B)$ with residue field
$\R$ has this form:
\end{lab}

\begin{lem}\label{chactalphaxi}%
Given $\alpha\in\Sper(A\otimes B)$, there exists $\xi\in V(\R)$ with
$\alpha=\alpha_\xi$ if and only if $(A\otimes B)/\supp(\alpha)=\R$.
\qed
\end{lem}

\begin{lab}\label{kprep2}%
We fix a \sa\ subset $K$ of $V(\R)$ and denote by $\wt K$ the
constructible subset of $\Sper(A)=\Sper\R[V]$ corresponding to~$K$,
see \ref{tilderappl}. The natural homomorphism $i\colon A\to A_B$
induces a continuous map $i^*\colon\Sper(A_B)\to\Sper(A)$ of the
real spectra (see \ref{sperrappl}), and we write $X_K:=(i^*)^{-1}
(\wt K)$.
So $X_K$ is a constructible subset of $\Sper(A_B)$, which is closed
in $\Sper(A_B)$ if $K$ is a closed subset of $V(\R)$. By $K_R$ we
denote the base field extension of $K$ to $R$ (see \cite{bcr} 5.1).
So $K_R$ is the \sa\ subset of $V(R)$ that is defined by the same
finite system of inequalities as $K$ (this does not depend on the
choice of such a system). Considering $V(R)$ as a subset of
$\Sper(A_B)$ in the natural way, we have $K_R=V(R)\cap X_K$ (c.f.\
\ref{tilderappl}).
\end{lab}

Recall that a closed point of a topological space $T$ is a point
$x\in T$ for which the singleton set $\{x\}$ is closed in~$T$.

\begin{prop}\label{remaxidrelat}%
Assume that the \sa\ set $K\subset V(\R)$ is compact. Then the closed
points of $X_K$ are precisely the points $\alpha_\xi$, for $\xi\in K$
(see~\ref{krprep}).
\end{prop}

\begin{proof}
For $\xi\in K$ we have $\alpha_\xi\in X_K$ by construction, and this
is a closed point of $\Sper(A_B)$ since $\supp(\alpha_\xi)=M_\xi$ is
a maximal ideal of $A_B$.
Conversely, let $\alpha\in X_K$ be a closed point of $X_K$, and let
$\phi\colon A\otimes B\to S$ be a homomorphism that represents
$\alpha$, where $S$ is a real closed field
(c.f.\ \ref{sperrappl}). Let $C\subset S$ be the convex hull of
$\R$ in $S$, so we have $C/\m_C=\R$. We claim that $\im(\phi)\subset
C$ holds. Indeed, let $a\in A$ and $b\in B$. Since $K$ is compact
there is $c\in\R$ with $|a|<c$ on $K$, and it follows that
$|\phi(a\otimes1)|<c$ in $S$.
On the other hand, there is a real number $c'>0$ such that $|b|<c'$
holds on $\Sper(B)$, for example $c'=1+|\ol b|$. So we get $|\phi
(a\otimes b)|<cc'$ in $S$, whence $\phi(a\otimes b)\in C$. Now since
$\im(\phi)\subset C$, we can compose $\phi\colon A\otimes B\to C$
with the residue homomorphism $C\to\R$, resulting in a homomorphism
$\psi\colon A\otimes B\to\R$. By construction, the point $\beta\in
\Sper(A\otimes B)$ represented by $\psi$ is a specialization of
$\alpha$. Since $K$ is closed in $V(\R)$ we have $\beta\in X_K$,
and so $\beta=\alpha$,
which proves the claim by Lemma \ref{chactalphaxi}.
\end{proof}

\begin{lem}\label{psdinox}%
Let $K\subset V(\R)$ be a \sa\ set, and let $f\in A_B$. Then $f$ is
nonnegative on the constructible subset $X_K$ of $\Sper(A_B)$ if, and
only if, $f$ is nonnegative on $K_R\subset V(R)$.
\end{lem}

\begin{proof}
Let $\q$ be a prime ideal of $B$. The quotient field $R(\q)$ of
$B/\q$ is real closed. Let $\pi_\q(f)\in A\otimes R(\q)=A_{R(\q)}$ be
the coefficient-wise reduction of $f$ modulo~$\q$. On the other hand,
let $K_{R(\q)}\subset V(R(\q))$ be the base field extension of $K$
from $\R$ to~$R(\q)$. Then $f\ge0$ on $X_K$ is equivalent to
$\pi_\q(f)\ge0$ on $K_{R(\q)}$ for every prime ideal $\q$ of~$B$.
Thus we have to show: If $f\ge0$ on $K_R\subset V(R)$, then
$\pi_\q(f)\ge0$ on $K_{R(\q)}$, for every prime ideal $\q$ of~$B$. To
see this, recall that the residue map $B_\q\to R(\q)$ has a
homomorphic section~$s$. Thus if $\eta\in K_{R(\q)}$ is a given
homomorphism $\eta\colon A\to R(\q)$, then $\xi:=s\comp\eta$,
considered as a homomorphism $A\to B_\q\subset R$, is a point in
$K_R$. Since $f\ge0$ at~$\xi$, it follows that $\pi_\q(f)\ge0$ at
$\eta$.
\end{proof}

\begin{lab}\label{spezreeldisk}%
Now we specialize to the case where $C$ is an irreducible affine
curve over $\R$, and $A=\R[C]$ is the affine coordinate ring of $C$.
We keep fixed the extension $\R\subset R$ of real closed fields and
the convex hull $B$ of $\R$ in $R$, and we'll write $R[C]:=A\otimes
R$ and $B[C]:=A\otimes B$. The following technical lemma is specific
to the curves case.
\end{lab}

\begin{lem}\label{spezreelrelat}%
Let $C$ be an irreducible affine curve over $\R$, and let $K\subset
C(\R)$ be a compact \sa\ set. Let $M$ be a maximal ideal of $\R[C]
\otimes B=B[C]$, and assume that there exists $\beta\in X_K$ with
$\supp(\beta)\subset M$ and with $\supp(\beta)\not\subset
\R[C]\otimes\m$. Then $M=M_\xi$ for some $\xi\in K$.
\end{lem}

\begin{proof}
Write $A=\R[C]$ as before. Let $P=\supp(\beta)$, write $\q=P\cap B$,
and let $k=R(\q)=B_\q/\q B_\q$ be the residue field of the prime
ideal $\q$ of $B$. The field $k$ is real closed.
The sequence of ring homomorphisms $B\to A\otimes B\to A\otimes k$
induces, by taking preimages, a sequence of maps
$$\Spec(A\otimes k)\>\labelto j\>\Spec(A\otimes B)\>\labelto\pi\>
\Spec(B)$$
between the Zariski spectra. The map $j$ is a bijection from
$\Spec(A\otimes k)$ to the preimage $\pi^{-1}(\q)=C_k$ of $\q$ under
$\pi$, and this bijection preserves residue fields of prime ideals.
Since $A\otimes k=k[C]$ is a one-dimensional integral domain, since
the zero ideal of $A\otimes k$ corresponds to $A\otimes\q\in\pi^{-1}
(\q)$, and since $P\not\subset A\otimes\m$ by assumption, we see that
$P$ corresponds to a maximal ideal of $A\otimes k$ under this
bijection.
Since moreover the residue field of $P$ is real, there exists a point
$\eta\in C(k)$ such that $P$ is the kernel of the homomorphism
$$A\otimes B\>=\>\R[C]\otimes B\>\to\>\R[C]\otimes k\>=\>k[C]\>
\labelto\eta\>k.$$
So $(A\otimes B)/P$ is isomorphic to a subring of $k$ that contains
the valuation ring $B/\q$ of~$k$. Therefore $(A\otimes B)/P$ is a
valuation ring itself, and in particular, a local ring. Therefore $M$
is the unique maximal ideal of $A\otimes B$ that contains $P$.
On the other hand, by Proposition \ref{remaxidrelat}, there exists
$\xi\in K\subset C(\R)$ such that $\beta$ specializes to
$\alpha_\xi$, and hence $P\subset M_\xi$.
This shows $M=M_\xi$.
\end{proof}

\begin{lab}
We keep fixing the extension $\R\subset R$ and the valuation ring $B$
of $R$ as before. We now assume that $C$ is a nonsingular and
geometrically irreducible affine algebraic curve over $\R$, and we
consider the affine scheme $C\times_{\Spec(\R)}\Spec(B)=
\Spec(\R[C]\otimes B)$. This is a relative affine curve over
$\Spec(B)$. If $B$ were a discrete valuation ring, the situation
would be a (very particular) instance of a relative curve over a
Dedekind scheme, hence an arithmetic surface. However, $B$ has
divisible value group and therefore is not noetherian (as long as
$R\ne\R$). Moreover, the Krull dimension of $B$ can be arbitrarily
large. Therefore we cannot directly rely on arguments that are
well-known for arithmetic surfaces, or simply for noetherian rings.
Still, the situation and the auxiliary results we are about to prove,
resemble the case of a relative curve over a discrete valuation ring.

The function field of $C$, resp.\ of $C_R$, is as usual denoted by
$\R(C):=\Quot\R[C]$, resp.\ by $R(C):=\Quot R[C]$.
\end{lab}

\begin{lab}\label{spezdfn}%
Let $R'=R(\sqrt{-1})$ be the algebraic closure of $R$, and let $B'=
B[\sqrt{-1}]$, a valuation ring of $R'$ that extends the valuation
ring $B$ of $R$. The maximal ideal of $B'$ will be denoted $\m'$, and
we have $B'/\m'=\C$.
The valuation $v$ on $R$ (see \ref{wsorit}), resp.\ $w$ on $R(C)$
(see \ref{geomintbew}), extends uniquely to a valuation on $R'$,
resp.\ on $R'(C)$, and we use the same letter $v$, resp.\ $w$, to
denote this extension.
The residue field of the valuation $v$ on $R'$ is $\C$, and the
residue field of the valuation $w$ on $R'(C)$ is $\C(C)$, the complex
function field of the curve $C$. Given $g\in R'(C)$ with $w(g)\ge0$,
we denote the residue class of $g$ in $\C(C)$ by $\ol g$. Also, we
write $B'[C]=\R[C]\otimes B'$ and $R'[C]=\R[C]\otimes R'$. Again we
have $B'[C]=\{f\in R'[C]\colon w(f)=0\}$.

We consider the natural specialization map
$$C(B')\>\to\>C(\C),\quad\eta\mapsto\ol\eta$$
defined by composing a homomorphism $\eta\colon\R[C]\to B'$ with the
residue map $B'\to B'/\m'=\C$. Note that $\eta\in C(B')$ specializes
to
$\xi\in C(\C)$ (that is, $\ol\eta=\xi$) if, and only if, $h(\eta)=0$
implies $\ol h(\xi)=0$, for every $h\in B'[C]$.
Given $\xi\in C(\C)$, we'll use the notation
$$U(\xi)\>:=\>\{\eta\in C(B')\colon\ol\eta=\xi\},$$
so this is the set of $B'$-rational points of $C$ that specialize to
the $\C$-rational point~$\xi$. The maximal ideal of $B[C]$ associated
with $\xi\in C(\C)$ is denoted $M_\xi=\{f\in B[C]\colon
\ol f(\xi)=0\}$, see \ref{krprep}.

The zero or pole order of a rational function $g$ on a nonsingular
curve in a geometric point $\xi$ will be denoted by $\ord_\xi(g)$.
Thus, given $f\in B'[C]$ and $\eta\in C(R')$, the symbol $\ord_\eta
(f)$ denotes the vanishing order of $f$ in the point $\eta$ of the
generic fibre $C_{R'}$. For $\xi\in C(\C)$, on the other hand, the
symbol $\ord_\xi(\ol f)$ denotes the vanishing order in $\xi$ of the
restriction $\ol f$ of $f$ to the special fibre $C$.
Below (Proposition \ref{nstordspez}) we show how the vanishing orders
of $f$ in points of the generic fibre determine the vanishing orders
of $\ol f$ in the points of the special fibre.
\end{lab}

\begin{lem}\label{ord0quotbewr}%
Let $g\in R(C)^*$ satisfy $w(g)=0$, let $\xi\in C(\C)$ be a geometric
point of the special fibre, and assume $\ord_\eta(g)\ge0$ for every
$\eta\in U(\xi)$. Then there exist $0\ne f,\,h\in B[C]$ with
$g=\frac fh$ and $\ol h(\xi)\ne0$. In other words, $g$ lies in the
localized ring $B[C]_{M_\xi}$.
\end{lem}

\begin{proof}
We can write $g=\frac ba$ with $0\ne a,\,b\in R[C]$. By scaling
$a$ and~$b$ with a nonzero element of $R$ we clearly can assume
$w(a)=w(b)=0$.
So in particular $a,\,b\in B[C]$.

Let $\eta_1,\dots,\eta_r$ be the zeros of $a$ in $U(\xi)$, and let
$\zeta_1,\dots,\zeta_s$ be the remaining zeros of $a$ in $C(R')$.
For each $j=1,\dots,s$ there exists $h_j\in B[C]$ satisfying
$\ol h_j(\xi)\ne0$ and $h_j(\zeta_j)=0$, since $\ol\zeta_j\ne\xi$.
By taking a product of suitable powers of these $h_j$, we find $h\in
B[C]$ satisfying $\ol h(\xi)\ne0$ and $\ord_{\zeta_j}(h)\ge
\ord_{\zeta_j}(a)$ for $j=1,\dots,s$.

For any point $\eta\in C(R')$ we claim that $\ord_\eta(bh)\ge
\ord_\eta(a)$ holds. Indeed, this is trivial if $a(\eta)\ne0$. For
$\eta\in\{\zeta_1,\dots,\zeta_s\}$ it is so by the choice of~$h$.
For $\eta\in\{\eta_1,\dots,\eta_r\}$ it is true since $\ord_\eta(b)
\ge\ord_\eta(a)$ by the assumption on $g$.
So $gh=\frac{bh}a$ has no poles in $C(R')$, and therefore lies in
$R[C]$. Since $w(gh)=0$,
we have $gh\in B[C]$, so it suffices to take $f:=gh$.
\end{proof}

The analogue of Lemma \ref{ord0quotbewr} in algebraic geometry would
be the following statement: If $V$ is a nonsingular complex algebraic
surface and $\xi\in V(\C)$, and if a rational function $g\in\C(V)^*$
has no pole along any curve $C\subset V$ through $\xi$, then
$g\in\scrO_{V,\xi}$. (Indeed, the noetherian local ring
$\scrO_{V,\xi}$, being integrally closed, is the intersection of its
localizations at all height one prime ideals.)

\begin{lem}\label{planecurvcase}%
Let $f,\,g\in B'[x,y]$ be polynomials such that the coefficient-wise
reduced polynomials $\ol f,\,\ol g\in\C[x,y]$ are not identically
zero. Assume $f(0,0)=g(0,0)=0$, and assume that the curves $\ol f=0$
and $\ol g=0$ in $\C^2$ intersect transversally at $(0,0)$. Then the
curves $f=0$ and $g=0$ in $R'^2$ intersect transversally at $(0,0)$,
and they do not intersect in any point $(a,b)\ne(0,0)$ in $R'^2$ with
$a,\,b\in\m'$.
\end{lem}

\begin{proof}
The gradient vectors of $f$ and $g$ at the origin lie in $B'^2$, and
by assumption they are linearly independent modulo $\m'$. Hence they
are linearly independent in $R'^2$, which is the first assertion.
After a linear change of coordinates we can assume
$$f\>=\>x+\sum_{d\ge2}f_d(x,y),\quad g\>=\>y+\sum_{d\ge2}g_d(x,y)$$
where $f_d,\,g_d\in B'[x,y]$ are homogenous polynomials of degree
$d$, for $d\ge2$. Let $(0,0)\ne(a,b)\in\m'\times\m'$, and assume
$v(a)\le v(b)$. Since $v(a)>0$ we see that $v(f(a,b)-a)>v(a)$,
whence $v(f(a,b))=v(a)$, and therefore $f(a,b)\ne0$. Likewise,
$v(a)\ge v(b)$ implies $v(g(a,b))=v(b)$ and $g(a,b)\ne0$.
\end{proof}

\begin{lem}\label{spezordx2x}%
Let $\eta\in C(B')$, let $\xi=\ol\eta\in C(\C)$.
\begin{itemize}
\item[(a)]
There is $s\in B'[C]$ such that $s(\eta)=0$ and $\ord_{\xi}(\ol s)=
1$.
\item[(b)]
If $\eta\in C(B)$ then an element $s$ satisfying (a) can be found in
$B[C]$.
\item[(c)]
For any element $s$ satisfying (a) one has $\ord_\eta(s)=1$ and
$s(\eta')\ne0$ for any $\eta'\in U(\xi)\setminus\{\eta\}$.
\end{itemize}
\end{lem}

\begin{proof}
Choose $t\in B'[C]$ such that $\ol t\in\C[C]$ is a local uniformizer
at $\xi=\ol\eta$. Then $t(\eta)\in\m'$. The element $s:=t-t(\eta)$ of
$B'[C]$ has $s(\eta)=0$ and $\ol s=\ol t$, hence
$\ord_{\ol\eta}(\ol s)=1$. If $\eta$ is real, i.e.\ $\eta\in C(B)$,
then $t$ (and therefore $s$) can be found in $B[C]$. This proves (a)
and (b).

(c)
The question is local around the point $\ol\eta\in C(\C)$. Zariski
locally around any given $\C$-point, any nonsingular curve over $\C$
is isomorphic to a Zariski open subset of a plane curve over~$\C$.
Therefore we can assume that $C$ is a (possibly singular) closed
curve in $\A^2_\C$, and that $\xi=(0,0)$ is a nonsingular point of
$C$. Now assertion (c) follows from Lemma \ref{planecurvcase}.
\end{proof}

\begin{prop}\label{nstordspez}%
Let $f\in B'[C]$ satisfy $w(f)=0$. The vanishing order of $\ol f$ in
a point $\xi\in C(\C)$ satisfies
$$\ord_\xi(\ol f)\>=\>\sum_{\eta\in U(\xi)}\ord_\eta(f).$$
\end{prop}

\begin{proof}
Let $e$ denote the right hand sum in the assertion, and let
$$\{\eta\in U(\xi)\colon\>f(\eta)=0\}\>=:\>\{\eta_1,\dots,\eta_r\},$$
a finite set of points in $U(\xi)\subset C(B')\subset C(R')$. For
every $i=1,\dots,r$, choose $s_i\in
B'[C]$ with $w(s_i)=0$, $s_i(\eta_i)=0$ and $\ord_\xi(\ol s_i)=1$,
according to Lemma \ref{spezordx2x}(a). Moreover, put $e_i:=
\ord_{\eta_i}(f)$. Let $s:=s_1^{e_1}\cdots s_r^{e_r}\in B'[C]$, then
we have $w(s)=0$ and $\ord_\xi(\ol s)=e_1+\cdots+e_r=e$. Moreover,
from Lemma \ref{spezordx2x}(c) we see that $\ord_{\eta_i}(s)=e_i
=\ord_{\eta_i}(f)$ for $i=1,\dots,r$, and $s(\eta)\ne0$ for any
$\eta\in U(\xi)\setminus\{\eta_1,\dots,\eta_r\}$. Hence the rational
function $g:=\frac fs\in R'(C)^*$ has $\ord_\eta(g)=0$ for any
$\eta\in C(B')$ with $\ol\eta=\xi$. Applying Lemma \ref{ord0quotbewr}
to $g$ and $g^{-1}$ shows that $g$ is a unit in the localized ring
$B'[C]_{M_\xi}$. Thus $\ol g(\xi)\ne0$, and therefore
$\ord_\xi(\ol f)=\ord_\xi(\ol s)=e$.
\end{proof}

For an analogue of Proposition \ref{nstordspez} in algebraic geometry
let $V$ be a nonsingular complex surface and $C\subset V$ an
irreducible curve. Given a rational function $f\in\C(V)^*$ of order
zero along $C$, the proposition corresponds to the formula for the
divisor of the restriction of $f$ to~$C$.


\section{Main theorem}\label{sect:main}%

The following fact is well-known:

\begin{thm}\label{satpofg}%
Let $C$ be a nonsingular affine curve over $\R$, and let $K\subset
C(\R)$ be a compact \sa\ set. Then the saturated preordering
$\scrP(K)$ of $K$ in $\R[C]$ is finitely generated.
\end{thm}

This is proved in \cite{sch:mz} Theorem 5.21. More precisely (by
\cite{sch:mz} 5.22(b)), $\scrP(K)$ can be generated by two elements
(even as a quadratic module), and can in fact be generated by a
single element whenever $K$ has no isolated points. If $K=C(\R)$
(assuming this set is compact) we have $\scrP(K)=\Sigma\R[C]^2$.

\begin{lab}\label{satpofgpf}%
Let us briefly indicate how Theorem \ref{satpofg} can be proved. (The
proof given in \cite{sch:mz} was more complicated since the
archimedean local-global principle was not yet available at that
time.)
When $K$ has no isolated points, then $\scrP(K)$ is generated by any
$f\in\scrP(K)$ which has simple zeros in the boundary points of $K$
and has no other zeros in $K$ (one can show that such $f$ exists).
This follows from the archimedean local-global principle (see Theorem
\ref{archlgp} below).
In the general case, let $\xi_1,\dots,\xi_r$ be the isolated points
of $K$. We modify the set $K$ by replacing each isolated point
$\xi_i$ with a small closed interval $[\xi_i,\eta_i]$ on $C(\R)$, for
which $\eta_i\ne\xi_i$ lies on the same connected component of
$C(\R)$ as $\xi_i$, and the intervall is so small that
$[\xi_i,\eta_i]\cap K=\{\xi_i\}$. Let $K_1$ be the modified set
obtained in this way, and note that $K_1$ has no isolated points. Let
$K_2$ be a second such modification of $K$ in
which $\xi_i$ gets replaced by $[\eta_i',\xi_i]$, where $\eta'_i\ne
\xi_i$ is again chosen close to $\xi_i$, but such that $\eta_i$ and
$\eta'_i$ lie on opposite sides of $\xi_i$ on the local branch of
$C(\R)$ around $\xi_i$. Then, by the first part of the argument,
there exists a single generator $f_j$ of $\scrP(K_j)$, for both $j=1,
2$. Again using the archimedean local-global principle, one concludes
that $\scrP(K)$ is generated by $f_1$ and $f_2$.
\end{lab}

The following theorem, resp.\ its corollary, is the first main
result of this paper:

\begin{thm}\label{main}%
Let $C$ be a nonsingular affine curve over $\R$, let $K\subset C(\R)$
be a compact \sa\ set, and let $T=\scrP(K)$ be the saturated
preordering of $K$ in $\R[C]$. For any real closed field $R$
containing $\R$, the preordering $T_R$ generated by $T$ in $R[C]$ is
saturated as well.
\end{thm}

Using the notion of stable preordering, see
\ref{stablrappl}, we can give the following equivalent formulation:

\begin{cor}\label{mainstab}%
For $C$ and $K$ as in \ref{main}, the preordering $\scrP(K)$ in
$\R[C]$ is stable.
\end{cor}

\begin{proof}
By \cite{sch:stab} Corollary 3.8, $T=\scrP(K)$ is stable if and only
if for every real closed field $R$ containing~$\R$, the preordering
$T_R$ is saturated in $R[C]$. So \ref{mainstab} is equivalent to
\ref{main}.
\end{proof}

\begin{rems}
For the following remarks assume that the nonsingular affine curve
$C$ is irreducible.
\smallskip

1.\
When $C$ is rational, the assertions of Theorem \ref{main} and
Corollary \ref{mainstab} are true regardless whether $K$ is compact
or not. More precisely, assume that $C$ is a nonsingular rational
affine curve, and let $K\subset C(\R)$ be any closed \sa\ subset.
Then the saturated preordering $\scrP(K)$ of $K$ in $\R[C]$ is
finitely generated, and is stable. This is well-known, and
essentially elementary.
\smallskip

2.\
When $C$ has genus one and $C(\R)$ is compact, \ref{main} and
\ref{mainstab} were proved for $K=C(\R)$ in \cite{sch:conv}. In all
other cases of positive genus, these results are new.
\smallskip

3.\
When $C$ is nonsingular of genus $\ge1$ and $K\subset C(\R)$ is a
closed \sa\ set that is not compact, two situations can occur. Either
$K$ is virtually compact (see \ref{virtcptrappl} below); in this case
we'll later prove that the above results remain true (Theorem
\ref{vcptmain} below). Or else $K$ fails to be virtually compact;
then it is known that the preordering $\scrP(K)$ fails to be finitely
generated (\cite{sch:mz} Theorem 5.21), and so the notion of
stability does not even make sense for it. See \ref{virtcptex} below
for both examples and non-examples of virtually compact sets.
\end{rems}

Before giving the actual proof of Theorem \ref{main}, we need some
preparations. First recall the archimedean local-global principle:

\begin{thm}\label{archlgp}%
\emph{(\cite{sch:surf} Corollary 2.10)}
Let $A$ be a ring containing $\frac12$, let $P$ be an archimedean
preordering in $A$, and let $f$ be an element of the saturation of
$P$. Then $f$ lies in $P$ if (and only if) $f$ lies in $P_\m$ for
every maximal ideal $\m$ of $A$.
\end{thm}

Here $P_\m$ is the preordering generated by $P$ in the localized ring
$A_\m$. See \ref{quadmodsorit} for the notions of archimedean
preordering and saturation.

\begin{lab}\label{proofstart}%
Let in the following $C$ be a nonsingular affine curve over $\R$, let
$K\subset C(\R)$ be a compact \sa\ subset and $T=\scrP(K)\subset
\R[C]$. Moreover let $R$ be a real closed field containing $\R$, and
let $B$ be the convex hull of $\R$ in $R$ (see \ref{wsorit}). We
shall work in the ring $B[C]=
\R[C]\otimes B$, and shall use the auxiliary results from
Section~\ref{sect:aux}. In particular, we use the notation introduced
there. Let $T_B$ be the preordering generated by $T$ in $B[C]$. The
saturation of $T_B$ consists of all $f\in B[C]$ with $f\ge0$
on~$X_K$.
\end{lab}

\begin{lem}\label{archmodringerw}%
The preordering $T_B$ in $B[C]$ is archimedean.
\end{lem}

\begin{proof}
Since $T=\scrP(K)$ is the saturated preordering in $\R[C]$ associated
with the compact set $K$, it is clear that $T$ is archimedean.
Let $f\in B[C]=\R[C]\otimes B$. Since $T-T=\R[C]$, we can write $f$
in the form $f=\sum_{i=1}^rf_i\otimes b_i$ with $f_i\in T$ and
$b_i\in B$ ($i=1,\dots,r$). Since $T$ is archimedean, there exists
$0<c_1\in\R$ with $c_1-f_i\in T$ ($i=1,\dots,r$). By the definition
of $B$ there exists $0<c_2\in\R$ with $b_i\le c_2$ in $R$ for every
$i$, and hence $c_2-b_i$ is a square in $B$ for $i=1,\dots,r$. We
conclude that
$$rc_1c_2-f\>=\>c_2\sum_{i=1}^r(c_1-f_i)\otimes1+\sum_{i=1}^rf_i
\otimes(c_2-b_i)$$
lies in $T_B$.
\end{proof}

By  $\interior(K_R)$ we denote the interior, relative to $C(R)$, of
the \sa\ subset $K_R$ of $C(R)$. The following technical lemma is
based on Lemma \ref{spezordx2x}.

\begin{lem}\label{chirurg}%
Let $\xi\in K$, and let $U(\xi)=\{\eta\in C(B')\colon\ol\eta=\xi\}$
as in \ref{spezdfn}. For every point $\eta\in U(\xi)$ there exists
an element $p_\eta\in T_B$ with $w(p_\eta)=0$, such that
$p_\eta(\eta)=0$ and
$$\ord_\xi(\ol p_\eta)\>=\>\begin{cases}1&\text{if }\eta\in C(R),\
\eta\notin\interior(K_R),\\2&\text{if }\eta\in C(R),\ \eta\in
\interior(K_R),\\2&\text{if }\eta\in C(R')\setminus C(R).
\end{cases}$$
Moreover, if $\eta=\xi$ and $\xi$ is an isolated point of $K$, there
exists a second element $p'_\xi\in T_B$ with the same properties as
$p_\xi$ and such that $p_\xi p'_\xi\le0$ on a neighborhood of $\xi$
in $C(R)$.
\end{lem}

\begin{proof}
We need to distinguish several cases. First assume $\eta\in C(R')
\setminus C(R)$.
By Lemma \ref{spezordx2x}(a) there exists $s\in B'[C]$ with
$s(\eta)=0$ and $\ord_\xi(\ol s)=1$.
Let $\tau$ be the $R$-automorphism of $R'[C]$ of order two that is
induced by complex conjugation on~$R'$. Then $p_\eta:=s\cdot\tau(s)$
is a sum of two squares in $B[C]$, hence $p_\eta\in T_B$, and clearly
$p_\eta(\eta)=0$ and $\ord_\xi(\ol p_\eta)=2$.

When $\eta\in\interior(K_R)$, choose $s\in B[C]$ with $s(\eta)=0$ and
$\ord_\xi(\ol s)=1$, according to Lemma \ref{spezordx2x}(b). Then
$p_\eta:=s^2$ will do the job.

Now assume $\eta\in C(R)$ and $\eta\notin\interior(K_R)$. Then
necessarily $\xi$ is a boundary (or isolated) point of $K$, and
either $\eta=\xi$ or $\eta\notin K_R$. Since $T$ is saturated, there
exists $t\in T$ with $\ord_\xi(t)=1$ and with $t(\eta)\le0$. (The
second condition is automatic if $\xi$ is not an isolated point of
$K_R$.)
So $p_\eta:=t-t(\eta)$ lies in $T_B$ and has the desired properties.

To prove the additional claim in the case where $\eta=\xi$ is an
isolated point of $K$, fix a local orientation on $C(\R)$ around
$\xi$. Since $T$ is saturated, one can find $t_1\in T$ changing sign
from $+$ to $-$ in $\xi$, as well as $t_2\in T$ changing sign from
$-$ to $+$ in $\xi$, such that both have vanishing order $1$ in
$\xi$.
The proof of the lemma is complete.
\end{proof}

\begin{proof}[Proof of Theorem \ref{main}]
We have to show that
$T_R$ contains every $g\in R[C]$ with $g\ge0$ on $K_R$. It suffices
to prove that $T_B$ contains every $f\in B[C]$ with $f\ge0$ on $K_R$
and with $w(f)=0$. Indeed, given $g\in R[C]$ with $g\ge0$ on $K_R$,
we find $0\ne b\in R$ with $w(g)=v(b^2)$, and hence with $b^{-2}g\in
B[C]$ and $w(b^{-2}g)=0$. Knowing $b^{-2}g\in T_B$ clearly implies
$g\in T_R$.

So fix $f\in B[C]$ with $f\ge0$ on $K_R$ and with $w(f)=0$.
From Lemma \ref{psdinox} we know that $f\ge0$ on $X_K$, that is,
$f$ lies in the saturation of $T_B$ in $B[C]$. Since $T_B$ is
archimedean (Lemma \ref{archmodringerw}), we can apply the
archimedean local-global principle \ref{archlgp} to $f$ and $T_B$. By
this theorem, it suffices to prove, for every maximal ideal $M$ of
$B[C]$, that $f$ lies in $T_M$, the preordering generated by $T$ in
the local ring $B[C]_M$. To show this fix $M$, and let
$X_{K,M}:=X_K\cap\Sper B[C]_M$, where $\Sper B[C]_M$ is considered as
a subset of $\Sper B[C]$ in the natural way. So $X_{K,M}$ is the
basic closed constructible subset of $\Sper B[C]_M$ associated with
$T_M$.

If $f>0$ on $X_{K,M}$, then $f\in T_M$ by \cite{sch:lgp1}
Proposition~2.1. So we can assume that there exists $\beta\in X_K$
with $f\in\supp(\beta)\subset M$. The hypotheses of Lemma
\ref{spezreelrelat} apply to $M$, since $w(f)=0$ implies $\supp
(\beta)\not\subset\R[C]\otimes\m$. By Lemma \ref{spezreelrelat},
therefore, we have $M=M_\xi$ for some point $\xi\in K$.
Recall that
$$U(\xi)\>=\>\{\eta\in C(B')\colon\>\ol\eta=\xi\}.$$
We decompose the set of $R$-zeros of $f$ in $U(\xi)$ as
$$\{\eta\in U(\xi)\cap C(R)\colon\>f(\eta)=0\}\>=\>\{\eta_1,\dots,
\eta_r\}\cup\{\zeta_1,\dots,\zeta_s\}$$
in such a way that $\eta_1,\dots,\eta_r$ are interior points of
$K_R$, while $\zeta_1,\dots,\zeta_s$ are not. Note that $f$ has even
order in any of the points $\eta_i$.
Among the nonreal zeros of $f$ in $U(\xi)$, choose a
subset $\{\omega_1,\dots,\omega_t\}$ that contains exactly one
representative from each pair of complex conjugate points. Then put
$$p\ :=\ \prod_{i=1}^r(p_{\eta_i})^{\frac12\ord_{\eta_i}(f)}\cdot
\prod_{j=1}^s(p_{\zeta_j})^{\ord_{\zeta_j}(f)}\cdot\prod_{k=1}^t
(p_{\omega_k})^{\ord_{\omega_k}(f)},$$
where the $p_{\eta_i}$, $p_{\zeta_j}$, $p_{\omega_k}\in T_B$ are
chosen as in Lemma \ref{chirurg}.
Then $p$, being a product of elements of $T_B$, lies in $T_B$. By
Proposition \ref{nstordspez} we have
$$\ord_\xi(\ol f)\>=\>\sum_{i=1}^r\ord_{\eta_i}(f)+\sum_{j=1}^s\ord
_{\zeta_j}(f)+2\sum_{k=1}^t\ord_{\omega_k}(f).$$
This number is also equal to $\ord_\xi(\ol p)$.
It follows that $g:=\frac fp$ is a unit in the local ring $B[C]_M$.
In particular, $g$ has no zeros or poles in $U(\xi)$.

We would like $g$ to take positive values in all points $\eta\in
U(\xi)\cap K_R$.
This obviously is the case whenever $p(\eta)\ne0$. By continuity, it
is also true whenever $\eta$ is not an isolated point of $K_R$. The
remaining case when $\eta\in U(\xi)$ is an isolated point of $K_R$
can occur only for $\eta=\xi\in C(\R)$, and when $\xi$ is an isolated
point of $K$. If $g(\xi)<0$, we replace one of the local factors
$p_\xi$ in the definition of $p$ by $p'_\xi$, where $p'_\xi$ is
chosen as in Lemma \ref{chirurg}. If $p'$ denotes the modification of
$p$ obtained in this way, and $g'=f/p'$, we have achieved
$g'(\xi)>0$.

Using Lemma \ref{psdinox} we see that the unit $g$ of $B[C]_M$ takes
strictly positive values on the set $X_{K,M}$ associated with the
preordering $T_M$.
Hence, by another application of \cite{sch:lgp1} Proposition 2.1, we
conclude that $g$ lies in $T_M$. As a consequence, it follows that
$f=pg\in T_M$, as desired. The proof of Theorem \ref{main} is
complete.
\end{proof}


\section{Semidefinite representations in the compact
  case}\label{sect:sdp}%

Now we use moment relaxation to obtain semidefinite representations
from the results of the previous section.

\begin{thm}\label{sdp}%
Let $K\subset\R^n$ be a compact convex \sa\ set whose set $\Ex(K)$ of
extreme points has (\sa) dimension~$\le1$. Then $K$ has a
semidefinite representation. Such a representation can be obtained
from a suitable moment relaxation.
\end{thm}

\begin{proof}
The closure $K_0:=\ol{\Ex(K)}$ is a compact \sa\ set and satisfies
$\dim(K_0)\le1$. We have $K=\conv(K_0)$ by the Krein-Milman theorem.
We may assume that $K$ is not contained in any proper affine-linear
subspace of $\R^n$.

Let $I\subset\R[\x]=\R[x_1,\dots,x_n]$ be the ideal of polynomials
vanishing on $K_0$, and let $C_0=\Spec(\R[\x]/I)$, so $C_0$ is the
reduced Zariski closure of $K_0$ in $\A^n$. Then $C_0$ is an
$\R$-variety (possibly reducible) of dimension~$\le1$. Let $\pi\colon
C'_0\to C_0$ be the normalization of $C_0$. Note that $\pi$ is a
finite morphism, and that $C'_0$ is nonsingular. If $C_0$ has the
irreducible components $X_1,\dots,X_l$, and if we denote by $A'_i$
the integral closure of $\R[X_i]$ in its quotient field, the
coordinate ring of $C'_0$ is therefore $\R[C'_0]=A'_1\times\cdots
\times A'_l$.

The map $\pi\colon C'_0(\R)\to C_0(\R)$ on $\R$-points
may fail to be surjective. Indeed, when $\xi$ is an isolated point of
$C_0(\R)$ that lies on a one-dimensional irreducible component of
$C_0$, then $\xi\notin\pi(C'_0(\R))$.
To resolve this problem, let $\xi_1,\dots,\xi_k$ be the isolated
points of $C_0(\R)$ that lie in $K_0$ and that lie on one-dimensional
irreducible components of $C_0$, and write $P_i=\Spec(\R)$ for
$i=1,\dots,k$. Finally let
$$C_1\>=\>C'_0\>\amalg\>P_1\>\amalg\>\cdots\>\amalg\>P_k$$
(disjoint sum), and let $\phi\colon C_1\to C_0$ be the morphism with
$\phi|_{C'_0}=\pi$ and with $\phi(P_i)=\xi_i$ ($i=1,\dots,k$). Let
$K_1$ be the preimage of $K_0$ in $C_1(\R)$. Since $\pi$, and
therefore also $\phi$, is a finite morphism, the \sa\ set $K_1$ is
compact. By construction we have $\phi(K_1)=K_0$.
Since $C_1$ is nonsingular with irreducible components of dimension
$\le1$, the saturated preordering $\scrP(K_1)$ of $K_1$ in $\R[C_1]$
is finitely generated (see Theorem \ref{satpofg}). By the main result
of the previous section (Corollary \ref{mainstab}), the preordering
$\scrP(K_1)$ is stable. Note that $K_1$ is a basic closed set since
$\dim(K_1)\le1$ (see for instance \cite{abr}, VI.5.1 and III.3.1).

The morphism $\phi\colon C_1\to C_0\subset\A^n$ induces a
homomorphism $\varphi\colon\R[\x]\to\R[C_1]$ of the coordinate rings.
Since $K$ was assumed not to be contained in a proper affine-linear
subspace, the restriction of $\varphi$ to $L:=\spn(1,x_1,\dots,x_n)$
is injective.
We consider $L$ as a linear subspace of $\R[C_1]$. Let $\Sigma=\Sigma
\R[C_1]^2$, and choose $1=h_0,\,h_1,\dots,h_r\in\R[C_1]$ with
$\scrP(K_1)=h_0\Sigma+\cdots+h_r\Sigma$. Since $\scrP(K_1)$ is
stable, there exists a tuple $W=(W_0,\dots,W_r)$ of
finite-dimensional linear subspaces $W_i\subset\R[C_1]$ such that
$L\cap\scrP(K_1)$ is contained in
$$M_W\>=\>\Sigma_{W_0}+h_1\Sigma_{W_1}+\cdots+h_r\Sigma_{W_r},$$
see \ref{relaxrappl}. By Corollary \ref{relaxcpt} this implies
that we have found a semidefinite representation for $\conv
(\phi(K_1))=K$.
\end{proof}

\begin{example}\label{isolpktbsp}%
To illustrate the construction in the proof of Theorem \ref{sdp}, let
us consider the (rational) affine curve $C_0$ with equation
$$y^2+x^2(x-1)(x-2)\>=\>0.$$
The set $C_0(\R)\subset\R^2$ is compact and has the origin as an
isolated point.
To construct a semidefinite representation for the convex hull $K$ of
$C_0(\R)$, we work in $A_1=A'_0\times\R$, where $A'_0$ is the
integral closure of $A_0=\R[C_0]$, i.e.,
$$A'_0\>=\>\R[x,z]/\bigl(z^2+(x-1)(x-2)\bigr)$$
(where $y=xz$). Using the elements $1=(1,1)$, $u=(x,0)$, $v=(z,0)$
and $e=(1,0)$ of $A_1$, we let $L=\spn(1,u,uv)$, $W=\spn(1,e,u,v)$
and $U=WW=\spn(1,e,u,v,u^2,uv)$. The relaxation for $K$ obtained from
this data is exact. Using the basis $1-e,e,u,v$ for $W$, we get $K$
as the set of all $(\xi,\eta)\in\R^2$ for which there exist $a,b,c\in
\R$ with
$$\begin{pmatrix}1-c&0&0&0\\0&c&\xi&a\\0&\xi&b&\eta\\0&a&\eta&
3\xi-b-2c\end{pmatrix}\>\succeq\>0.$$
For the reader's convenience
we include the details of the argument: Since $v^2=-u^2+3u-2e$
we get, for
$$\mu=\mu_1+c\mu_e+x\mu_u+a\mu_v+b\mu_{u^2}+y\mu_{uv}\>\in U'$$
a general linear form, the matrix
$$M\>=\>M(x,y,a,b,c)\>=\>\begin{pmatrix}1-c&0&0&0\\0&c&x&a\\0&x&b&y\\
0&a&y&3x-b-2c\end{pmatrix}$$
with respect to the basis $1-e,\,e,\,v$ of $W$. This matrix represents
the pull-back of $\mu$ to a symmetric bilinear form on $W$, via the
product map $W\times W\to U$. Exactness of the relaxation is shown as
follows. Let $S=\{(\xi,\eta)\colon\ex a,b,c$ $M(\xi,\eta,a,b,c)
\succeq0\}$, let $K$ be the convex hull of $C_0(\R)$. The inclusion
$C_0(\R)\subset S$ is obvious. To prove $S\subset K$, let
$M(\xi,\eta,a,b,c)\succeq0$. Then $0\le c\le1$. Exploiting the
$2\times2$ minors $M_{23}$ and $M_{34}$ we get
$\eta^2+\xi^2(\xi^2-3\xi+2c)\le0$. This implies $(\xi,\eta)\in K$
when $c=0$ or $c=1$. Let $0<c<1$. Since the right bottom $3\times3$
submatrix of $M$ is homogeneous, we can scale with $\frac1c$ and get
$M(\frac\xi c,\frac\eta c,\frac ac,\frac bc,1)\succeq0$. By what was
just remarked we have $(\frac\xi c,\frac\eta c)\in K$, and hence
$(\xi,\eta)\in K$ as well.
\end{example}

\begin{rem}
It was already mentioned that the dimension hypothesis $\dim(K)\le1$
in Theorem \ref{sdp} is essential, according to \cite{sch:sdp}.
Similarly, this hypothesis is also essential for the stability result
Theorem \ref{mainstab}, from which Theorem \ref{sdp} was derived.
Indeed, there does not exist any compact \sa\ set $K\subset\R^n$ with
$\dim(K)\ge2$ such that the saturated preordering $\scrP(K)$ is
finitely generated and stable. This follows from the main result of
\cite{sch:stab}.
\end{rem}


\section{Semidefinite representations in the general
  case}\label{sect:noncpt}%

Using the compact case, we now establish semidefinite representations
for the closed convex hulls of arbitrary one-dimensional \sa\ sets,
and will deduce the dimension two case of the Helton-Nie conjecture.
I am indebted to Tim Netzer who showed me how to obtain semidefinite
representations for noncompact closed convex sets from such
representations for compact sets.

\begin{thm}\label{sdpclosed}%
Let $K\subset\R^n$ be the closed convex hull of a \sa\ set of
dimension $\le1$. Then $K$ has a semidefinite representation.
\end{thm}

\begin{lab}\label{convprep}%
Before we start the proof, we need to recall a few notions on convex
sets and cones, for which we refer to \cite{r}, Theorems 8.1 and 8.2.
Given a nonempty closed convex set $K\subset\R^n$, the
\emph{recession cone} of $K$ is
$$\rc(K)\>=\>\{x\in\R^n\colon\>K+x\subset K\},$$
and is a closed convex cone. Note that $\rc(K)$ can also be described
as the set of all existing limits $\lim_{\nu\to\infty}a_\nu x_\nu$ in
$\R^n$, where $x_\nu$ is a sequence in $K$ and $a_\nu$ is a null
sequence of positive real numbers.
The \emph{homogenization} $K^h$ of $K$ is the closure of the convex
cone $K^c=\{(t,tx)\colon t\ge0$, $x\in K\}$ in $\R\times\R^n=
\R^{n+1}$, and is described as
$$K^h\>=\>K^c\cup\{(0,y)\colon y\in\rc(K)\}.$$
The original set $K$ is recovered from its homogenization as
$K=\{x\in\R^n\colon(1,x)\in K^h\}$. The extreme rays of $K^h$ are
the rays spanned by points $(1,x)$ with $x\in\Ex(K)$, together with
the rays spanned by points $(0,y)$ where $\R_\plus y$ is an extreme
ray of $\rc(K)$.
\end{lab}

\begin{lab}\label{asydir}%
Let $S\subset\R^n$ be a \sa\ set. A ray $\R_\plus u$ (with $0\ne u\in
\R^n$) will be called an \emph{asymptotic direction of $S$ at
infinity} if there exist continuous \sa\ paths $a(t)$ in $\R$ and
$x(t)$ in $S$ (with $0<t\le1$) such that $a(t)>0$, $a(t)\to0$ and
$a(t)x(t)\to u$ for $t\to0$.
\end{lab}

\begin{prop}\label{rcconvhull}%
Let $S\subset\R^n$ be a nonempty closed \sa\ set, and let $K=
\ol{\conv(S)}$ be its closed convex hull.
\begin{itemize}
\item[(a)]
Each extreme point of $K$ is contained in $S$.
\item[(b)]
Each extreme ray of $\rc(K)$ is an asymptotic direction of $S$ at
infinity.
\end{itemize}
\end{prop}

Without the hypothesis that $S$ is \sa, assertion (a) remains
certainly true as long as $S$ is bounded, but we are not sure about
the general case.

\begin{proof}
For the proof of both parts we can assume $\rc(K)\cap(-\rc(K))=
\{0\}$. (Otherwise $K$ contains a line, which implies that $\Ex(K)=
\emptyset$ and $\rc(K)$ has no extreme ray.)
We are first going to show that for any $\xi\in K$ there exists
$u\in\rc(K)$ with $\xi-u\in\conv(S)$; note that this implies
$\Ex(K)\subset\conv(S)$, and hence~(a).
Let $\xi\in K$. By the curve selection lemma and by Carath\'eodory's
lemma, there exist
continuous \sa\ paths $a_i(t)$ in $[0,1]$ and $x_i(t)$ in $S$, for
$i=0,\dots,n$ and $0<t\le1$, such that $\sum_{i=0}^na_i(t)\equiv1$,
and such that
\begin{equation}\label{limpath}%
x(t)\>=\>\sum_{i=0}^na_i(t)x_i(t)
\end{equation}
converges to $\xi$ for $t\to0$. Note that the limit $\alpha_i:=\lim
_{t\to0}a_i(t)$ exists in $[0,1]$ for every $0\le i\le n$ since the
functions $a_i(t)$ are \sa, and that $\sum_i\alpha_i\equiv1$.

We claim that the curves $a_i(t)x_i(t)$ ($i=0,\dots,n$) are bounded
for $t\to0$. Indeed, assume that $a_i(t)x_i(t)$ is unbounded for at
least one index~$i$. Since the $a_i(t)x_i(t)$ have Puiseux Laurent
expansions in $t$ for small $t>0$, we see that there exists a minimal
rational number $q>0$
such that, for every $0\le i\le n$, the curve $t^qa_i(t)x_i(t)$ is
bounded and therefore the limit $u_i:=\lim_{t\to0}t^qa_i(t)x_i(t)$
exists in $\R^n$.
Then $u_i\in\rc(K)$ for every $i$,
and $u_i\ne0$ for at least one index~$i$.
Multiplying \eqref{limpath} with $t^q$ shows $\sum_{i=0}^nu_i=0$,
contradicting $\rc(K)\cap(-\rc(K))=\{0\}$.

So the curves $a_i(t)x_i(t)$ are all bounded. Hence the limits $u_i:=
\lim_{t\to0}a_i(t)x_i(t)$ exist in $\R^n$. If $x_i(t)$ is unbounded
then $\alpha_i=0$ and $u_i\in\rc(K)$.
If $x_i(t)$ is bounded then $\xi_i=\lim_{t\to0}x_i(t)$ exists in $S$,
and $u_i=\alpha_i\xi_i$. Let $y$ denote the sum of the $u_i$ for
those indices $i$ for which $x_i(t)$ is bounded, and let $u$ be the
sum of the remaining $u_i$. Then $y\in\conv(S)$,
$u\in\rc(K)$ and $\xi=y+u$. This proves our assertion, and hence~(a).

The proof of (b) is similar. After making a
translation we can assume $0\in K$. Let $0\ne u\in\rc(K)$. Similar to
\eqref{limpath} we have
$$\frac1tu-w(t)\>=\>\sum_{i=0}^na_i(t)x_i(t)$$
($0<t\le1$) with \sa\ paths $a_i(t)$ in $[0,1]$ and $x_i(t)$ in $S$,
where $\sum_ia_i(t)\equiv0$ and $w(t)$ is a correction term with
$|w(t)|<1$.
Multiplication with $t$ gives
$$u-tw(t)\>=\>\sum_{i=0}^nta_i(t)x_i(t).$$
For $t\to0$, the summands on the right remain bounded, as shown
above.
Therefore the limit $u_i=\lim_{t\to0}ta_i(t)x_i(t)$ exists in $\R^n$
for $i=0,\dots,n$, and $\R_\plus u_i$ is an asymptotic direction of
$S$ at infinity (see \ref{asydir}) if $u_i\ne0$.
From $u=\sum_iu_i$ we see that if $\R_\plus u$ is an extreme ray of
$\rc(K)$, then $\R_\plus u=\R_\plus u_i$ for some $i$, which proves
(b).
\end{proof}

\begin{lab}\label{pfsdpclosed}%
We now give the proof of Theorem \ref{sdpclosed}. Let $S\subset\R^n$
be a nonempty \sa\ set of dimension at most one, and let $K=
\ol{\conv(S)}$ be its closed convex hull. In order to prove that $K$
has a semidefinite representation we may assume $\rc(K)\cap(-\rc(K))
=\{0\}$. (Indeed, $U=\rc(K)\cap(-\rc(K))$ is a linear subspace of
$\R^n$, and $K+U\subset K$. If $\pi\colon\R^n\to\R^n/U$ is the
quotient map, then $\pi(K)=\ol{\pi(S)}$, and the recession cone $R$
of $\pi(K)$ satisfies $R\cap(-R)=\{0\}$. A semidefinite
representation for $\pi(K)$ immediately gives one for $K$.)
For the homogenization $K^h\subset\R^{n+1}$ of $K$ (see
\ref{convprep}) this implies $K^h\cap(-K^h)=\{0\}$.
So the dual cone $(K^h)^*$ of $K^h$ in $\R^{n+1}$ is
full-dimensional,
and we can pick an interior point $w$ of $(K^h)^*$. The convex set
$$K_1\>:=\>\{x\in K^h\colon\>\bil xw=1\}$$
is compact,
and $K^h$ is (isomorphic to) the homogenization of $K_1$. Indeed,
since $\bil yw>0$ for $0\ne y\in K^h$, we have $K^h=\{tx\colon
x\in K_1$, $t\ge0\}$, and the right hand set is closed, hence equal
to $K_1^h$.

The extreme rays of the convex cone $K^h$ correspond to the extreme
points of $K$ and to the extreme rays of $\rc(K)$, see
\ref{convprep}. By Proposition \ref{rcconvhull}(a), $\Ex(K)\subset
\ol S$ has dimension $\le1$. The set $S$ has only finitely many
asymptotic directions at infinity since $\dim(S)\le1$, and so
$\rc(K)$ has only finitely many extreme rays by \ref{rcconvhull}(b).
Considering the set of extreme rays of $K^h$ as a subset of the unit
sphere in $\R^{n+1}$, this set therefore has dimension $\le1$. It
follows that the set $\Ex(K_1)$ of extreme points of $K_1$ has
dimension $\le1$ as well. So we can apply Theorem \ref{sdp} to $K_1$,
and conclude that $K_1$ has a semidefinite representation. By Lemma
\ref{naivconesdp} below, this implies that the cone $(K_1)^c=(K_1)^h
\cong K^h$ (first equality holds since $K_1$ is compact)
has a semidefinite representation as well. This completes the proof
of Theorem \ref{sdpclosed}, since $K$, being an affine-linear section
of $K^h$, also has a semidefinite representation.
\end{lab}

\begin{lem}\label{naivconesdp}%
Let $K\subset\R^n$ be a convex set. If $K$ has a semidefinite
representation, the same is true for the convex cone $K^c\subset
\R\times\R^n$ (see \ref{convprep}).
\end{lem}

\begin{proof}
This is certainly well-known: Assume
$$K\>=\>\{x\in\R^n\colon\ex y\in\R^m\ A+M(x)+N(y)\succeq0\}$$
where $M(x)$, $N(y)$ are linear systems of
symmetric matrices. Then $K^c$ is the set of all $(t,x)\in
\R\times\R^n$ for which there is $(s,y)\in\R\times\R^m$ with
$$tA+M(x)+N(y)\>\succeq\>0,\quad\begin{pmatrix}t&x_i\\x_i&s
\end{pmatrix}\>\succeq\>0\quad(i=1,\dots,n).$$
\end{proof}

The proof of Theorem \ref{sdpclosed} is therefore complete. We can
easily extend the theorem to closed conic hulls:

\begin{cor}\label{closedconichull}%
Let $S\subset\R^n$ be a \sa\ set, let $S_1:=\{\frac x{|x|}\colon
0\ne x\in S\}$ be its radial projection to the $(n-1)$-sphere. If
$\dim(S_1)\le1$ then the closed conic hull $\ol{\cone(S)}$ of $S$ has
a semidefinite representation.
\end{cor}

\begin{proof}
Here $\cone(S)$, the convex cone generated by $S$, consists of all
finite linear combinations of elements of $S$ with non-negative
coefficients. For the proof consider $K:=\ol{\conv(S_1)}$, a compact
convex set in $\R^n$ that has a semidefinite representation by
Theorem \ref{sdpclosed}. By Lemma \ref{naivconesdp}, the cone
$K^c=K^h\subset\R\times\R^n$ of $K$ has a semidefinite representation
as well. Since $\ol{\cone(S)}$ is the closure of the projection of
$K^c$ to $\R^n$, the assertion of the corollary follows.
\end{proof}

Now we combine Theorem \ref{sdpclosed} with results of Netzer to
show:

\begin{thm}\label{hndim2}%
\emph{(Helton-Nie conjecture in dimension two)}
Every convex \sa\ subset of $\R^2$ has a semidefinite
representation.
\end{thm}

\begin{proof}
Let $K\subset\R^2$ be a convex \sa\ set. To prove that $K$ has a
semidefinite representation, we first consider the case when $K$ is
closed. If $K$ contains a line, the assertion is obvious by reduction
to a (closed) convex subset of $\R$.
So we assume $\rc(K)\cap(-\rc(K))=\{0\}$. Then $K=\conv\Ex(K)+\rc(K)=
\ol{\conv\Ex(K)}+\rc(K)$ (Minkowski sum, \cite{r} Theorem 18.5),
and the set $\Ex(K)$ is \sa\ of dimension $\le1$. By Theorem
\ref{sdpclosed}, $\ol{\conv\Ex(K)}$ is sdp-representable.
Since $\rc(K)$ is clearly sdp-representable, being a closed convex
cone in $\R^2$,
we see that $K$ is sdp-representable as well.

Now let $K\subset\R^2$ be an arbitrary convex \sa\ set. We can assume
that $K$ has nonempty interior. Let $M$ be the set of points in the
boundary $\partial K=\partial\ol K$ that do not lie in $K$.
Then $M$ is a \sa\ set with $\dim(M)\le1$, and we can decompose $M$
set-theoretically as follows. Let $M_0$ be the relative topological
interior of $M\cap\Ex(\ol K)$ inside $\partial\ol K$, and let $\scrF$
be the set of one-dimensional faces of $\ol K$. The supporting line
of every $F\in\scrF$ is an irreducible component of the Zariski
closure of $\partial K$. Therefore the set $\scrF$ is finite.
For each $F\in\scrF$, let $M_F=F\cap M$. Moreover, let $H_F$ be the
open halfplane with $H_F\cap K\ne\emptyset$ whose boundary line
contains $F$, and let $K_F=H_F\cup(F\cap K)=H_F\cup(\ol H_F\cap K)$.
Then $M$ is the union of $M_0$ with finitely many extreme points of
$\ol K$ and with $\bigcup_{F\in\scrF}M_F$.
Accordingly, $K$ is the intersection of $K_0:=\ol K\setminus M_0$
with finitely many sets $K_\xi:=\ol K\setminus\{\xi\}$ (where $\xi\in
\Ex(\ol K)$) and with the sets $K_F$ ($F\in\scrF$).

Since a finite intersection of sdp-representable sets is again
sdp-representable, it suffices to show that each of $K_0$, $K_\xi$
and $K_F$ as above is sdp-representable. Each of the
sets $K_F$ is a union of an open halfplane $H$ with a convex subset
of the line $\partial H$. Using the result of Netzer and Sinn
\cite{ns}, such $K_F$ has a semidefinite representation. (Due to the
elementary nature of this situation, one can easily find an explicit
such representation directly.)
The sets $K_\xi$ ($\xi\in\Ex(\ol K)$) have semidefinite
representations by \cite{nt} Proposition 3.1.
For $K_0$ we employ Netzer's construction from \cite{nt}. Let $N=
\partial\ol K\setminus M_0$, a closed subset of $\partial\ol K$ with
$K_0=\interior(K)\cup N$, and let $T=\ol{\conv(N)}$. Then $T$ is a
closed convex subset of $\ol K$, and is sdp-representable by Theorem
\ref{sdpclosed}. By construction, and by Proposition
\ref{rcconvhull}, $T\cap\partial\ol K=N=\partial\ol K\setminus M_0$.
Using the notation introduced in \cite{nt}, let $(T\lal\ol K)$ denote
the union of the relative interiors of all the faces of $\ol K$ that
meet $T$. We see that $(T\lal\ol K)=\interior(K)\cup N=K_0$. By
\cite{nt} Theorem 3.8, $(T\lal\ol K)$ is sdp-representable,
which proves our theorem.
\end{proof}

Homogenizing, we see that the Helton-Nie conjecture holds for convex
cones in $\R^3$:

\begin{cor}\label{hnconer3}%
Every \sa\ convex cone $C\subset\R^3$ has a semidefinite
representation.
\end{cor}

\begin{proof}
We may assume $C\cap(-C)=\{0\}$.
In fact we easily reduce to the case where $C\ne\{0\}$ and there
exists $u\in\R^3$ with $\bil ux>0$ for every $0\ne x\in C$.
Let $L:=\{x\in\R^3\colon\bil ux=1\}$. Then $K:=C\cap L$ has a
semidefinite representation by Theorem \ref{hndim2}. Since
$C=\{tx\colon x\in K$, $t\ge0\}$ is a linear image of the cone $K^c$
and $K^c$ has a semidefinite representation by Lemma
\ref{naivconesdp}, we are done.
\end{proof}


\section{Stability in the virtually compact case}%
\label{sect:virtcpt}%

\begin{lab}\label{virtcptrappl}%
Let $C$ be an irreducible affine curve over $\R$, and let $K\subset
C(\R)$ be a closed \sa\ subset. Adopting the terminology of
\cite{sch:mz}, \cite{sch:guide}, we say that $K$ is \emph{virtually
compact} if there exists a nonconstant regular function $f\in\R[C]$
that is bounded on $K$. Equivalently, $K$ is virtually compact if and
only if there exists an irreducible affine curve $C_1$ containing $C$
as a Zariski open subset, in such a way that the points in
$C_1\setminus C$ are nonsingular on $C_1$ and the closure $K_1$ of
$K$ in $C_1(\R)$ is compact.

When the affine curve $C$ is not necessarily irreducible, a closed
\sa\ set $K\subset C(\R)$ is called virtually compact if $K\cap
C'(\R)$ is virtually compact on $C'$, for every irreducible component
$C'$ of~$C$. A closed \sa\ set $K\subset\R^n$ of dimension $\le1$ is
called virtually compact if it has this property with respect to its
Zariski closure $C$.
\end{lab}

\begin{examples}\label{virtcptex}%
A closed \sa\ set $K\subset\R$ is virtually compact only if it is
compact. For more interesting examples let $C$ be an irreducible
plane curve with equation $f(x,y)=0$. If the highest degree
homogeneous part of $f$ has a nonreal linear factor, then every
closed \sa\ set $K\subset C(\R)$ is virtually compact.
For yet another class of examples consider plane curves $C$ with
equation $y^2=p(x)$, where $p\in\R[x]$ is monic and separable with
$\deg(p)=d$.
If $d=2$ or $d$ is odd, only compact sets $K\subset C(\R)$ are
virtually compact.
If $d\equiv0$ (mod~$4$), then $K\subset C(\R)$ is virtually compact
iff $K$ is contained in the union of a bounded set with either the
upper or the lower halfplane. If $d\equiv2$ (mod~$4$), $d\ge6$, a
similar characterization holds with upper or lower halfplanes
replaced by the unions of diagonally opposite quadrants.
\end{examples}

We show that the analogues of the stability results from Section
\ref{sect:main} remain true for virtually compact sets $K$:

\begin{thm}\label{vcptmain}%
Let $C$ be an irreducible nonsingular affine curve over $\R$, and let
$K\subset C(\R)$ be a closed \sa\ set that is virtually compact. Then
the saturated preordering $\scrP(K)$ in $\R[C]$ is finitely generated
and stable.
\end{thm}

\begin{proof}
That $\scrP(K)$ is finitely generated was already proved (in greater
generality) in \cite{sch:mz} Theorem 5.21. We are going to reprove
this fact here using a different reasoning, because we'll need the
same argument to prove stability. Since Theorem \ref{vcptmain} has
already been proved when $K$ is compact, we can assume that $K$ is
not compact. In particular, the set $K$ is infinite.

Let $C_1$ and $K_1$ be as in \ref{virtcptrappl}. Then $C_1$ is a
nonsingular irreducible affine curve containing $C$ as a Zariski open
set, and the closure $K_1$ of $K$ in $C_1(\R)$ is compact. Note that
$\R[C_1]$ is a
subring of $\R[C]$. Let $T=\scrP_C(K)$, the saturated preordering of
$K$ in $\R[C]$, and let $T_1=\scrP_{C_1}(K_1)$, the saturated
preordering of $K_1$ in $\R[C_1]$. Since $K_1$ is compact, the
preordering $T_1$ in $\R[C_1]$ is finitely generated according to
Theorem \ref{satpofg}.
So there are nonzero elements $1=h_0,\,h_1,\dots,h_r\in\R[C_1]$ that
generate $T_1$ as a quadratic module in $\R[C_1]$. (We can even do
with $r\le2$, see the remarks after Theorem \ref{satpofg}.) We'll
prove that $T=\scrP_C(K)$ is generated by $h_0,\dots,h_r$ as a
quadratic module in $\R[C]$.

Let $\wt C$ be the nonsingular projective curve over $\R$ that
contains $C_1$ as an open dense subscheme. We consider Weil divisors
on $\wt C$, and regard them as conjugation-invariant Weil divisors on
the complexified curve $\wt C_\C$. Since $\wt C(\R)\ne\emptyset$,
we have $\Pic(\wt C)=\Pic(\wt C_\C)^{\Gal(\C/\R)}$.
Let $J$ be the Jacobian variety of $\wt C$, an abelian variety over
$\R$.

Let $C_1(\C)\setminus C(\C)=\{Q_1,\dots,Q_s\}$,
and let $0\ne f\in
\R[C]$ with $f|_K\ge0$. For $i=1,\dots,s$ let $m_i\ge0$ be an integer
satisfying $2m_i+\ord_{Q_i}(f)\ge0$. Consider the divisor
$$D\>=\>\sum_{i=1}^sm_iQ_i$$
on $\wt C$. Choose a point $Q\in\wt C(\C)\setminus C_1(\C)$, and let
$E=Q+\ol Q$ (again a divisor on $\wt C$, the case $Q=\ol Q$ is
allowed, bar denoting complex conjugation). There exist integers
$l$, $n\ge1$ such that the divisor $nE-lD$ has degree zero, and such
that the divisor class $[nE-lD]\in J(\R)$ lies in the identity
connected component $J(\R)_0$ of the compact real Lie group $J(\R)$.
Fix an arbitrary $\R$-point $P_0$ in the interior $\interior(K)$ of
$K$ relative to $C(\R)$. By the argument in \cite{sch:tams}, 2.11 and
2.12, there is an integer $k\ge1$ such that, for every $\alpha\in
J(\R)_0$, there exist $2k$ points $P_1,\dots,P_{2k}\in\interior(K)$
with
$$\alpha\>=\>\sum_{j=1}^{2k}[P_j-P_0].$$
Applying this to the divisor class $\alpha:=[nE-lD-k(2P_0-E)]$ (which
lies in $J(\R)_0$, c.f.\ \cite{sch:tams} Lemma 2.6),
we conclude that there exist $P_1,\dots,P_{2k}\in\interior(K)$ such
that
$$lD+\sum_{j=1}^{2k}P_j\>\sim\>(n+k)E$$
on $\wt C$.
Since $\supp(E)$ is disjoint to $C_1$, there exists $0\ne
h\in\R[C_1]$ such that the divisor of $h$ on $C_1$ is $lD+\sum_{j=1}
^{2k}P_j$.
Since $\ord_{Q_i}(h^2f)\ge2lm_i+\ord_{Q_i}(f)\ge0$, we see that
$h^2f$ lies in $\R[C_1]$ as well. Moreover, every zero of $h$ on $C$
is real and is an interior point of $K$.
In addition, we can ensure that $h$ has no common zero with any of
$h_0,\dots,h_r$.

Since $f\ge0$ on $K$, and since $K$ is dense in $K_1$, it follows
that $h^2f\ge0$ on $K_1$. So $h^2f\in T_1$, which means that there is
an identity
$$h^2f\>=\>\sum_{i=0}^r\sum_jp_{ij}^2h_i$$
with suitable $p_{ij}\in\R[C_1]$. Since any zero of $h$ on $C$ is real
and is an interior point of $K$, it follows that each summand
$p_{ij}^2h_i$ of the right hand sum is divisible (inside $\R[C]$) by
$h^2$, see \cite{sch:tams} Lemma 0.1. By the choice of $h$, none of
the $h_i$ vanishes in any of the zeros of $h$. Hence we even have
$h\mid p_{ij}$ inside $\R[C]$, for all indices $i,\,j$. Dividing we
conclude that $f$ lies in the quadratic module generated by $h_0,
\dots,h_r$ in $\R[C]$.

We have thus proved that $T=\scrP_C(K)$ is finitely generated in
$\R[C]$. To prove that $T$ is stable is equivalent to proving the
following assertion (c.f.\ \cite{sch:stab} Corollary 3.8): Let $R$
be any real closed extension field of $\R$. Then the preordering
$T_R$ generated by $T$ in $R[C]$ is saturated.

To prove this, let $0\ne f\in R[C]$ be nonnegative on $K_R$, where
$K_R$ denotes the extension of the \sa\ set $K\subset C(\R)$ to a
\sa\ subset of $C(R)$. Arguing literally as in the first part of the
proof, we find $h\in\R[C_1]$ (sic) such that $h^2f\in R[C_1]$, and
such that any zero of $h$ on $C$ is real and is an interior point of
$K$ in which $h_0\cdots h_r$ does not vanish. Completing the argument
exactly as before, we see that $f$ lies in the quadratic module of
$R[C]$ generated by $h_0,\dots,h_r$. In other words, $f\in T_R$, as
desired. The theorem is proved.
\end{proof}


\end{document}